\documentclass{amsart}
\usepackage{amsmath}
\usepackage{amssymb}
\usepackage{amsthm}
\usepackage{amscd}
\usepackage[all]{xy}

\newtheorem{prop}{Proposition}
\newtheorem{thm}[prop]{Theorem}
\newtheorem{cor}[prop]{Corollary}
\newtheorem{lem}[prop]{Lemma}

\newtheorem{claim}[prop]{Claim}
\newtheorem{quest}{Question}

\theoremstyle{definition}
\newtheorem*{defi}{Definition}
\newtheorem{rem}[prop]{Remark}
\newtheorem{example}{Example}

\numberwithin{prop}{section}
\numberwithin{example}{section}
\numberwithin{equation}{section}

\DeclareMathOperator{\VZ}{VZ}
 \DeclareMathOperator{\Ht}{Onc}
 \DeclareMathOperator{\Comm}{Comm}
 \DeclareMathOperator{\Inn}{Inn}
 \DeclareMathOperator{\chrc}{char}
 
 \DeclareMathOperator{\Aut}{Aut}
 \DeclareMathOperator{\Out}{Out}
 \DeclareMathOperator{\Zen}{Z}
 \DeclareMathOperator{\TAut}{TAut}
 \DeclareMathOperator{\image}{im}
 \DeclareMathOperator{\kernel} {ker}
 \DeclareMathOperator{\iid} {id}

 \DeclareMathOperator{\Cent} {Cent}
 \DeclareMathOperator{\ab} {ab}
 \DeclareMathOperator{\GL} {GL}
 \DeclareMathOperator{\SL} {SL}
 \DeclareMathOperator{\PGL} {PGL}
 \DeclareMathOperator{\PSL} {PSL}
 \DeclareMathOperator{\rk} {rk}
 
 \DeclareMathOperator{\sep} {sep}
 
 \DeclareMathOperator{\VAut} {VAut}
 \DeclareMathOperator{\Iso} {Iso}
 
 \DeclareMathOperator{\scd} {scd}
 \DeclareMathOperator{\cd} {cd}
 \DeclareMathOperator{\vcd} {vcd}
 \DeclareMathOperator{\Gal} {Gal}

\newcommand{\ON}{\mathcal{ON}}

\newcommand{\ca}[1]{\mathcal{#1}}
\newcommand{\eu}[1]{\mathfrak{#1}}

\newcommand{\bo}[1]{\mathbf{#1}}
\newcommand{\rat}[2]{{#1}(X_{1},\ldots,X_{#2})}

\newcommand{\boA}{\mathbf{A}}
\newcommand{\boF}{\mathbf{F}}
\newcommand{\Com}{\mathbf{Com}}
\newcommand{\interm}{\mathbf{Int}}
\newcommand{\bE}{\bo{FGSep}}

\newcommand{\Z}{\mathbb{Z}}
\newcommand{\Q}{\mathbb{Q}}
\newcommand{\F}{\mathbb{F}}
\newcommand{\RR}{\mathbb{R}}
\newcommand{\N}{\mathbb{N}}
\newcommand{\Fp}{{\mathbb F_p}}

\newcommand{\bg}{\bar{g}}
\newcommand{\hGamma}{\hat{\Gamma}}

\newcommand{\nott}{\ca{N}}
\newcommand{\Fpt}{\F_p[\![t]\!]}
\newcommand{\RFpt}{\F_p(\!(t)\!)}

\newcommand{\deq}{\colon\!=}
\newcommand{\argu}{\hbox to 7truept{\hrulefill}}

\newcommand{\la}{\langle}
\newcommand{\ra}{\rangle}

\newcommand{\Gamhat}{{\widehat{\Gamma}}}
\newcommand{\dbQ}{{\mathbb Q}}
\newcommand{\grL}{{\mathfrak L}}

\newcommand{\dbG}{\mathbb{G}}

\newcommand{\IC}{{\rm ICom}}

\newcommand{\dbA}{\mathbb{A}}
\newcommand{\dbZ}{\mathbb{Z}}

\newcommand{\Autbar}{\underline{\rm Aut\,}}
\begin{document}

\title[]
{Abstract commensurators of profinite groups}
\author{Yiftach Barnea}
\address{
Department of Mathematics\\
   Royal Holloway, University of London\\
   Egham, Surrey TW20 0EX\\
   United Kingdom}
\email{y.barnea@rhul.ac.uk}
\author{Mikhail Ershov}
\address{
University Of Virginia\\
   Department of Mathematics\\
   P.O. Box 400137\\
   Charlottesville, VA 22904-4137\\
   United States of America}
\email{ershov@virginia.edu}
\author{Thomas Weigel}
\address{
Universit\`a degli Studi di Milano-Bicocca\\
Dipartimento di Matematica e Applicazioni\\
Via R. Cozzi, 53\\
I-20125 Milan, Italy}
\email{thomas.weigel@unimib.it}


\thanks{First published in Trans. Amer. Math. Soc. 363 (2011), 5381--5417, published by the American Mathematical Society}
\thanks{\copyright American Mathematical Society}

\maketitle
\begin{abstract}
In this paper we initiate a systematic study of the abstract commensurators of profinite groups.
The abstract commensurator of a profinite group $G$ is a group $\Comm(G)$ which depends only on the commensurability class of $G$.
We study various properties of $\Comm(G)$; in particular, we find two natural ways to turn it into a topological group.
We also use $\Comm(G)$ to study topological groups which contain $G$ as an open subgroup
(all such groups are totally disconnected and locally compact).
For instance, we construct a topologically simple group which contains the
pro-$2$ completion of the Grigorchuk group as an open subgroup. On the other hand, we show that some profinite groups
cannot be embedded as open subgroups of compactly generated topologically simple groups.
Several celebrated rigidity theorems, like Pink's analogue of Mostow's strong rigidity theorem for simple algebraic groups defined
over local fields and the Neukirch-Uchida theorem, can be reformulated
as structure theorems for the commensurators of certain profinite groups.
\end{abstract}

\section{Introduction}

Let $G$ be a group and let $H$ be a subgroup $G$. The (relative) commensurator of $H$ in $G$,
denoted $\Comm_G(H)$, is defined as the set of all $g\in G$ such that the group $gHg^{-1}\cap H$
has finite index in both $H$ and $gHg^{-1}$. This notion proved to be fundamental
in the study of lattices in algebraic groups over local fields and automorphism
groups of trees (see \cite{Margulis:book}, \cite{BassLub} and references therein).

The concept of an abstract commensurator is a more recent one. A virtual automorphism
of a group $G$ is defined to be an isomorphism between two finite index subgroups of $G$;
two virtual automorphisms are said to be equivalent if they coincide on some finite index subgroup of $G$.
Equivalence classes of virtual automorphisms are easily seen to form a group, called the abstract commensurator
(or just {\it the commensurator} of $G$) and denoted $\Comm(G)$. If $G$ is a subgroup of a larger group $L$,
there is a natural map $\Comm_L(G)\to \Comm(G)$ which is injective under some natural conditions, so
$\Comm(G)$ often contains information about all relative commensurators.

In this paper we study commensurators of profinite groups. If $G$ is a profinite group,
the commensurator $\Comm(G)$ is defined similarly to the case of abstract groups,
except that finite index subgroups are replaced by open subgroups, and virtual automorphisms
are assumed to be continuous. Our main goal in this paper is to develop the general theory of commensurators of
profinite groups and to apply this theory to the study of totally disconnected locally compact groups.

\subsection{Totally disconnected locally compact groups and the universal property of $\mathbf{Comm(G)}$}
Recall that profinite groups can be characterized as totally disconnected compact groups;
on the other hand, by van Dantzig's theorem \cite{dan:stu} every totally disconnected locally compact (t.d.l.c.) group
contains an open compact subgroup (which must be profinite). If $G$ is a profinite group,
by an {\it envelope } of $G$ we mean any topological group $L$ containing $G$
as an open subgroup. Thus, t.d.l.c. groups can be thought of as envelopes of
profinite groups.

Given a profinite group $G$, can one describe all envelopes of $G$?
This very interesting question naturally leads to the problem of computing $\Comm(G)$.
Indeed, if $L$ is an envelope of $G$, then for every $g\in L$ there exists an open
subgroup $U$ of $G$ such that $gUg^{-1}\subseteq G$; note that $gUg^{-1}$ is also
an open subgroup of $G$. Thus, conjugation by $g$ determines a virtual automorphism of $G$,
and we obtain a canonical homomorphism $L\to \Comm(G)$. The kernel of this map
is equal to $\VZ(L)$, the virtual center of $L$ (see \S\ref{ss:vircen}).
Under additional assumptions on $L$, e.g. if $L$ is topologically simple and compactly generated,
one has $\VZ(L)=\{1\}$. Thus, if $\Comm(G)$ is known, it becomes easier to describe envelopes of $G$.

\subsection{Commensurators of algebraic groups and rigid envelopes} Let $L$ be a t.d.l.c. group, and let
$G$ be an open compact subgroup of $L$. Generalizing the argument in the previous
paragraph, we obtain a canonical homomorphism $\kappa_{L}\colon\Aut(L)\to\Comm(G)$, and one might ask when
$\kappa_L$ is an isomorphism (this is entirely determined by $L$, not by $G$,
since if $G'$ is another open compact subgroup of $L$, then $\Comm(G')$ is canonically isomorphic
to $\Comm(G)$).
We will say that $L$ is {\it rigid } if every isomorphism between open compact subgroups of $L$
extends uniquely to an automorphism of $L$. It is easy to see that $\kappa_L$
is an isomorphism whenever $L$ is rigid, and the converse is true provided
$\VZ(L)=\{1\}$.

A large class of rigid groups is provided by the celebrated paper of Pink~\cite{pink}.
According to \cite[Cor.~0.3]{pink}, if $F$ is a non-archimedean local field
and $\dbG$ is an absolutely simple simply-connected algebraic group over $F$,
then the group of rational points $\dbG(F)$ is rigid. Thus, if $G$ is an open compact
subgroup of $\dbG(F)$, then $\Comm(G)$ is canonically isomorphic to $\Aut(\dbG(F))$.
For instance, if $\dbG=\SL_n$, we can take $G=\SL_n(O)$ where $O$ is the ring of integers
in $F$, so $\Comm(\SL_n(O))$ is isomorphic to $\Aut(\SL_n(F))$. It is well-known
that $\Aut(\SL_2(F))\cong \PGL_2(F)\rtimes \Aut(F)$, and
$\Aut(\SL_n(F))\cong \PGL_n(F)\rtimes (\Aut(F)\times \la d\ra)$
for $n\geq 3$ where $d$ is the Dynkin involution.

Rigidity has an interesting consequence in the case of topologically simple groups.
In Section~\ref{s:comm}, we will show that every topologically simple
rigid t.d.l.c. group $L$ can be canonically recovered from any of its open compact subgroups.
By Pink's theorem, this result applies to $L=\dbG(F)/\Zen(\dbG(F))$,
where $\dbG$ and $F$ are as in the previous paragraph, and $\Zen(\dbG(F))$ is the finite
center of $\dbG(F)$. It would be interesting to know which of the currently known topologically
simple t.d.l.c. groups are rigid. For instance, we believe that topological Kac-Moody groups
over finite fields are rigid; at the same time, we will show that there exists
a non-rigid topologically simple t.d.l.c. group (see Corollary~\ref{cor:nonrigid}).

\subsection{Topologically simple envelopes}
The following fundamental problem was formulated in a recent
paper of Willis~\cite{george:sim}:

\begin{quest}
\label{quest:Wil1}
Let $L_1$ and $L_2$ be topologically simple t.d.l.c. groups.
Suppose that there exist open compact subgroups $G_1$ of $L_1$ and $G_2$ of $L_2$
such that $G_1$ is isomorphic to $G_2$. Is $L_1$ necessarily isomorphic to $L_2$?
\end{quest}
For our purposes, it is convenient to reformulate this problem as follows:
\begin{quest}
\label{quest:Wil2}
Let $L$ be a topologically simple t.d.l.c. group, and let
$G$ be an open compact subgroup of $L$. Is it true that any topologically simple envelope
of $G$ is isomorphic to $L$?
\end{quest}

We do not have the answer to this question in general, but it is already interesting
to know what happens for a specific group $L$. Using Pink's theorem,
we give a positive answer to Question~\ref{quest:Wil2} when $L=\dbG(F)/\Zen(\dbG(F))$
for some absolutely simple simply-connected algebraic group $\dbG$ and a local field $F$
(see Proposition~\ref{prop:unique}).

Now let $G$ be a profinite group which does not have a ``natural'' topologically
simple envelope. In this case, the basic question is not the uniqueness,
but the existence of a topologically simple envelope. There are two groups for which
this question is particularly interesting: the Nottingham group and the profinite
completion of the first Grigorchuk group.

Recall that the Nottingham group $\mathcal N(F)$ over a finite field $F$ is the group
of wild automorphisms of the local field $F((t))$. It is well-known that $\mathcal N(F)$
enjoys many similarities with Chevalley groups over $F[[t]]$. Furthermore, in \cite{mikh:new},
it was shown that $\mathcal N(F)$ is a product of finitely many subgroups each of which
can be thought of as a non-linear deformation of $\SL_2^1(F[[t]])$, the first congruence
subgroup of $\SL_2(F[[t]])$ (assuming $\chrc F>2$). Since the group
$\SL_2^1(F[[t]])$ has the natural topologically simple envelope $\PSL_2(F((t)))$,
it was very interesting to know if there is an analogous
envelope $L$ for the Nottingham group. If such $L$ existed, one would expect it to be topologically simple.
In \cite{Klopsch:aut}, Klopsch proved that $\Aut (\mathcal N(F))$ is a finite extension of $\mathcal N(F)$,
and in \cite{ersh:nott} it is shown that $\Comm (\mathcal N(F))\cong \Aut(\mathcal N(F))$
for $F=\Fp$ where $p>3$ is prime. Thus, $\Comm (\mathcal N(\Fp))$ is a profinite group for $p>3$.
This easily implies that $\mathcal N(\Fp)$ does not have any ``interesting'' envelopes;
in particular, it does not have any topologically simple envelopes.

Let $\Gamma$ be the first Grigorchuk group. In \cite{claas:comm}, R\"over proved
that $\Comm(\Gamma)$ is an (abstractly) simple group. This result suggests
that $\hGamma$, the profinite completion of $\Gamma$, may have a topologically simple
envelope. In this paper, we confirm this conjecture (see Theorem~\ref{thm:grig});
more precisely, we show that the subgroup of $\Comm(\hGamma)$ generated by $\Comm(\Gamma)$ and $\hGamma$
is a topologically simple envelope of $\hGamma$. We believe that this construction
yields a new example of a topologically simple t.d.l.c. group; furthermore, we will show
that this group is not rigid (as defined earlier in the introduction).

So far we discussed the problems of existence and uniqueness of topologically simple envelopes
for specific profinite groups. Are there general obstructions for the existence of a topologically
simple envelope, that is, can one prove that some profinite group $G$ does not have a topologically
simple envelope without computing $\Comm(G)$? This question becomes easier to answer if we
restrict our attention to compactly generated envelopes. In \cite{george:sim}, Willis
has shown that a solvable profinite group cannot have a compactly generated topologically
simple envelope. In this paper, we use commensurators to obtain several results of a similar flavour
(see Theorem~\ref{thm:rest} and Corollary~\ref{cor:stupid}). However, there are many interesting
cases where our criteria do not apply. For instance, we do not know if a finitely generated non-abelian free pro-$p$ group
has any topologically simple envelopes.

\subsection{The commensurator as a topological group} The structure of the commensurator of a profinite group
is easier to understand if we consider the commensurator as a topological group. In this paper
we introduce two topologies on $\Comm(G)$
-- the strong topology and the $\Aut$-topology -- which will serve different purposes.

The strong topology on $\Comm(G)$ is a convenient technical tool in the study of envelopes of $G$;
in particular, we show that $\Comm(G)$ with the strong topology plays the role of a universal
envelope of $G$, provided $\VZ(G)=\{1\}$. However, the corresponding topological structure on $\Comm(G)$
tells us little about the complexity of $\Comm(G)$ as a group.
From this point of view, a more adequate topology on $\Comm(G)$ is the $\Aut$-topology,
which is a natural generalization of the standard topology on the automorphism group
of a profinite group. In many examples where $\Comm(G)$ turns out to be isomorphic to a
familiar group, the $\Aut$-topology on $\Comm(G)$  coincides with the ``natural'' topology,
and in all these examples $\Comm(G)$ with the $\Aut$-topology is locally compact.
In general, local compactness of $\Comm(G)$ turns out to be equivalent to ``virtual stabilization'' of
the automorphism system of $G$. We show that some ``large'' profinite groups such as free
pro-$p$ groups and some branch groups do not satisfy this condition, and thus their commensurators
with the $\Aut$-topology are not locally compact. In all examples where $\Comm(G)$ with the $\Aut$-topology
is not locally compact, it seems very hard to describe $\Comm(G)$ itself and
the possible topologically
simple envelopes of $G$; however, non-local compactness of $\Comm(G)$ does impose an interesting
restriction on envelopes of $G$: it implies that $G$ does not have a second countable
topologically simple rigid envelope (see Proposition~\ref{prop:rigid_rest}).

\subsection{Commensurators of absolute Galois groups}
Let $F$ be a field, and let $F^{\sep}$ be a
{\it separable closure} of $F$. Then $F^{\sep}/F$ is a Galois extension, and the group
\begin{equation}
\label{eq:absgr}
G_F=\Gal(F^{\sep}/F)=\Aut_F(F^{\sep})
\end{equation}
is called the {\it absolute Galois group of $F$}. It carries canonically the structure of a profinite group.

In Section \ref{s:galois} we show that the Neukirch-Uchida theorem and its generalization by Pop
-- two important theorems in algebraic number theory -- can be interpreted as deep structure theorems
about the commensurators of certain absolute Galois groups.
The Neukirch-Uchida theorem is equivalent to the fact that the canonical map
$\iota_{G_{\Q}}\colon G_{\Q}\to\Comm(G_{\Q})_S$ is an isomorphism
(see Theorem~\ref{thm:neuc2}), where $\Comm(G_{\Q})_S$ denotes $\Comm(G_{\Q})$ with strong topology.

In order to give a reinterpretation of Pop's generalization of the Neukirch-Uchida theorem we introduce
certain totally disconnected locally compact groups $\{G_F(n)\}_{n\geq 0}$,
which generalize the absolute Galois group of $F$ in a natural way; in particular, $G_F(0)=G_F$.
We believe that these groups are of independent interest.
They satisfy a weak form of the Fundamental Theorem in Galois theory
(see Theorem~\ref{thm:main}), and as t.d.l.c. groups they have a very
complicated and rich structure which we do not discuss here any further.
Using Pop's theorem we show that for a field $F$ which is finitely generated
over $\Q$ of transcendence degree $n$ there is a canonical isomorphism between
$G_{\Q}(n)$ and $\Comm(G_F)_S$ (see Theorem~\ref{thm:popup}).

It is somehow surprising that the situation for $p$-adic fields seems to be much more
complicated. Mochizuki's version of the Neukirch-Uchida theorem for finite extensions of $\Q_p$
can be reinterpreted as a characterization of elements in  $\Comm(G_{\Q_p})_S$
which are contained in $\image(\iota_{G_F})$ for some finite extension $F/\Q_p$.
This suggests that the structure of $\Comm(G_{\Q_p})_S$ should be related
to the ramification filtrations on $G_{\Q_p}$. However, apart from some properties
which are related to the Galois cohomology of $p$-adic number fields, the structure
of $\Comm(G_{\Q_p})_S$ remains a mystery to the authors.

\vskip .1cm
{\bf Acknowledgements.} The first and last authors
would like to thank Claas R\"over whose explanation of his work on commensurators
initiated their interest in the subject.
We are grateful to Andrei Jaikin-Zapirain and Andrei Rapinchuk
for helpful discussions and suggestions which resulted in improvement of several
results in this paper.

\section{Preliminaries}
\label{s:norm}

\subsection{The virtual center}
\label{ss:vircen} Let $L$ be a topological group.
The subgroup \begin{equation}
\label{eq:defvircen}
\VZ(L)=\{\,g\in L\mid \Cent_L(g)\ \text{is open in $L$}\,\}.
\end{equation}
will be called the {\it virtual center} of $L$.\footnote{To the best of our knowledge,
the group $\VZ(L)$ was first introduced by Burger and Mozes~\cite{bumo:ihes1}
in the case of groups $L$ acting on a locally finite graph. This group
is denoted by ${\mathrm QZ}(L)$ in ~\cite{bumo:ihes1}.}
The following properties of $\VZ(L)$ are straightforward:
\begin{prop}
\label{prop:vircenbas}
Let $L$ be a topological group.
\begin{itemize}
\item[(a)] $\VZ(L)$ is a (topologically) characteristic subgroup of $L$.
\item[(b)] If $U$ is an open subgroup of $L$, then $\VZ(U)=\VZ(L)\cap U$.
\end{itemize}
\end{prop}

While the center of a Hausdorff topological group $G$ is always closed,
the virtual center $\VZ(G)$ need not be closed even if $G$ is a finitely
generated profinite group. For instance, let $\{S_n\}_{n\geq 1}$ be
pairwise non-isomorphic non-abelian finite simple groups, and let
$G=\prod_{n\geq 1} S_n$.
Then $G$ is a $2$-generated profinite group (see \cite{wilson:prof}),
and $\VZ(G)=\bigoplus_{n\geq 1} S_n$ is the direct sum of the subgroups $\{S_n\}$.
Hence $\VZ(G)$ is dense in $G$ and not closed.
The following proposition characterizes countably based profinite groups whose
virtual center is closed.

\begin{prop}
\label{prop:vircenprof}
Let $G$ be a countably based profinite group. Then $\VZ(G)$ is closed if and only if
for some open subgroup $U$ of $G$ one has $\VZ(U)=\Zen(U)$.
\end{prop}

\begin{proof}
Assume that there exists an open subgroup $U$ of $G$ such that $\VZ(U)=\Zen(U)$.
Then $\VZ(U)=\VZ(G)\cap U$ is closed and has also finite index in
$\VZ(G)$. This shows the `if' part of the proposition.

Assume that $\VZ(G)$ is closed, and
let $\ca{C}$ be a countable base for $G$.
Then $\VZ(G)=\bigcup_{W\in\,\ca{C}} \Cent_G(W)$. By  Baire's category
theorem, there exists $V\in\ca{C}$ such that $\Cent_G(V)$ is open in $\VZ(G)$ and thus has finite index in $\VZ(G)$.
Since $\Cent_G(U)\supseteq \Cent_G(V)$ whenever $U\subseteq V$, we conclude that
$\VZ(G)=\Cent_G(U)$ for some open subgroup $U$ of $G$. Then we have
$\VZ(U)= \VZ(G)\cap U=\Cent_G(U)\cap U=\Zen(U)$.
\end{proof}

\subsection{Continuous automorphisms of topological groups}
\label{ss:autotop}
Let $L$ be a topological group. By $\Aut(L)$ we denote the group of continuous automorphisms of $L$.
For $g\in L$ let $i_g\in\Aut(L)$ be the left conjugation by $g$,
that is,
\begin{equation}
\label{eq:aut2}
i_g(x) = gxg^{-1}\ \ \text{for all $g,x\in L$.}
\end{equation}
Let $i=i_L\colon L\to\Aut(L)$ be the canonical morphism given by $g\mapsto i_g$,
and let $\Inn(L)=\image(i)$, the subgroup of inner automorphisms of $L$.

In order to turn $\Aut(L)$ into a topological group, we need to make
additional assumptions on $L$.
First assume that $L$ has a base of neighborhoods of $1_L$ consisting of open subgroups.
In this case we can define the {\it strong topology} on $\Aut(L)$ using the following well-known
principle (see \cite{bou:top2}).

\begin{prop}
\label{prop:topbou}
Let $X$ be a group and let $\ca{F}$ be a set of subgroups of $X$.
Suppose that
\begin{itemize}
\item[(i)] for every $A,B\in\ca{F}$ there exists $C\in\ca{F}$ such that
$C\subseteq A\cap B$.
\item[(ii)]  for every $A\in\ca{F}$ and $g\in X$ there exists
$B\in\ca{F}$ such that $B\subseteq g^{-1}Ag$.
\end{itemize}
Then there exists a unique topology $\ca{T}_{\ca{F}}$ on $X$ with the property
that $(X,\ca{T}_{\ca{F}})$ is a topological group, and
$\ca{F}$ is a base of neighborhoods of $1_X$ in $\ca{T}_{\ca{F}}$.
\end{prop}

Let $\ca{F}$ be a base of neighborhoods of $1_L$
consisting of open subgroups of $L$.
By Proposition~\ref{prop:topbou},
$i(\ca{F}):=\{i(U) \mid\, U\in F\}$ is a base for unique topology $\ca{T}_S$ on $\Aut(L)$
which we call the {\it strong topology}. We will denote
the topological group $(\Aut(L),\ca{T}_S)$ by $\Aut(L)_S$.
Note that the induced topology on $\Out(L)=\Aut(L)/\Inn(L)$ is the discrete topology.
If $L$ is Hausdorff, then $\Zen(L)$ is closed, and thus
$\Aut(L)_S$ is also Hausdorff.

If $G$ is a profinite group, there is another natural topology on $\Aut(G)$,
which makes $\Aut(G)$ a profinite group, provided $G$ is finitely generated.
This topology (referred to as {\it standard topology } below)
will be discussed in Section \ref{ss:topaut}.

\subsection{The group of virtually trivial automorphisms}
\label{ss:atriv}
A continuous automorphism $\phi$ of a topological group $L$
will be called {\it virtually trivial }
if $\phi$ fixes pointwise some open subgroup of $L$. The set of all
virtually trivial automorphisms of $L$ will be denoted by $\TAut(L)$,
and it is clear that $\TAut(L)$ is a subgroup of $\Aut(L)$.
The following properties of $\TAut(L)$ are also straightforward:
\begin{prop}
\label{prop:taut}
Let $L$ be a topological group.
\begin{itemize}
\item[(a)] $\TAut(L)$ is a normal subgroup of $\Aut(L)$.
\item[(b)] $\TAut(L)\cap \Inn(L)= i(\VZ(L))$.
\end{itemize}
\end{prop}

It follows from Proposition \ref{prop:taut}(b) that for
a Hausdorff topological group $L$, the subgroup $\TAut(L)$ is closed in
$\Aut(L)_S$ if and only if $\VZ(L)$ is closed in $L$.
Furthermore, if $\VZ(L)$ is trivial, then so is $\TAut(L)$:

\begin{prop}
\label{prop:jiatriv}
Let $L$ be a topological group with trivial virtual center.
Then $\TAut(L)=\{1\}$.
\end{prop}

Proposition~\ref{prop:jiatriv} is a special case of a more general result:
\begin{prop}
\label{prop:isot}
Let $L$ be a topological group with trivial virtual center, and let
$\phi\colon U\to V$ be a topological isomorphism between open subgroups of $L$
such that $\phi\vert_W=\iid_W$ for some open subgroup
$W\subseteq U\cap V$.
Then $U=V$ and $\phi=\iid_U$.
\end{prop}

\begin{proof} Let $g\in U$. For every $x\in W\cap g^{-1}Wg$
we have $\phi(x)=x$ and $\phi(gxg^{-1})=gxg^{-1}$, and therefore
$[x,g^{-1}\phi(g)]=1$. Since $W\cap g^{-1}Wg$ is open in $L$,
we conclude that $g^{-1}\phi(g)\in \VZ(L)=\{1\}$.
Thus, we showed that $\phi(g)=g$ for every $g\in U$.
\end{proof}

\section{The commensurator of a profinite group}
\label{s:comm}

Let $G$ be a profinite group. A topological isomorphism from an open subgroup of $G$
to another open subgroup of $G$ will be called a {\it virtual automorphism }of $G$.
The set of all virtual automorphisms of $G$ will be denoted by $\VAut(G)$.
Two elements of $\VAut(G)$ are said to be equivalent, if they coincide
on some open subgroup of $G$. Equivalence classes of elements of $\VAut(G)$
form a group $\Comm(G)$ which we will call the {\it commensurator}  of the profinite group $G$.
More precisely, if $\phi\colon U\to V$ and $\psi\colon U'\to V'$
are two virtual automorphisms, and $[\phi], [\psi]\in \Comm(G)$ are the corresponding
equivalence classes, then $[\phi]\cdot[\psi]=[\theta]$ where
$\theta=\phi\vert_{U\cap V'}\circ\psi\vert_{\psi^{-1}(U\cap V')}$.

\begin{rem} If $G$ is a finitely generated profinite group, then all finite index
subgroups of $G$ are open by the remarkable recent theorem of Nikolov and Segal~\cite{nikseg1, nikseg2},
formerly known as Serre conjecture. Thus, in this case $\Comm(G)\cong \Comm(G_{abs})$
where $G_{abs}$ is $G$ considered as an abstract group.
\end{rem}

For every open subgroup $U$ of $G$ one has two canonical homomorphisms
\begin{equation}
\label{eq:canmap}
\begin{aligned}
\iota_U\colon& U&\longrightarrow&& \Comm(G),\\
\rho_U\colon& \Aut(U)& \longrightarrow&&\Comm(G).
\end{aligned}
\end{equation}
We put $\Autbar(U)=\image(\rho_U)$
and will usually write $\iota(U)$ instead of $\iota_U(U)$.
Note that $\kernel(\iota_U)=\VZ(U)$ and $\kernel(\rho_U)=\TAut(U)$.

Every virtual automorphism $\phi\in \VAut(U)$ can also be considered
as a virtual automorphism of $G$.
This correspondence yields a canonical mapping
\begin{equation}
\label{eq:ij}
j_{U,G}\colon\Comm(U)\longrightarrow\Comm(G).
\end{equation}
which is easily seen to be an isomorphism. Henceforth, we will usually identify
$\Comm(U)$ with $\Comm(G)$,
without explicitly referring to the isomorphism $j_{U,G}$.

\subsection{The commensurator of a profinite group as a topological group}
\label{ss:commtop}
There are two useful ways of topologizing
the commensurator of a profinite group.
The two topologies will be called the {\it strong topology} and the {\it Aut-topology}.
In this section we will define the strong topology and show how to use it
as a tool in studying relationship between totally disconnected locally compact (t.d.l.c.)
groups and their open compact subgroups. The Aut-topology, which is a natural
generalization of the standard topology on the automorphism group of a finitely generated
profinite group, will be defined in Section~\ref{ss:topaut}.

Let $G$ be a profinite group. The {\it strong topology} on $\Comm(G)$ can be defined as
the direct limit topology associated to the family of maps
$\{\,\iota_U \colon U\to\Comm(G)\mid \text{$U$ open in $G$}\,\}$, that is, the strongest topology on $\Comm(G)$ such that
all the maps $\iota_U$ are continuous. For our purposes, it will be more
convenient to give a more explicit definition. This definition
is unambiguous by Proposition~\ref{prop:topbou}.

\begin{defi} The {\it strong topology } $\ca{T}_S$ on $\Comm(G)$ is the unique topology
such that $(\Comm(G),\ca{T}_S)$  is a topological group and
the set $\{\,\image(\iota_U)\mid \text{$U$ open in $G$}\,\}$ is a base of neighborhoods
of $1_{\Comm(G)}$. We denote the topological group $(\Comm(G),\ca{T}_S)$ by $\Comm(G)_S$.
\end{defi}

\begin{prop}
\label{prop:strong}
Let $G$ be a profinite group.
\begin{itemize}
\item[(a)] If $U$ is an open subgroup of $G$, the canonical map
$j_{U,G}\colon\Comm(U)_S\to\Comm(G)_S$
is a homeomorphism.
\item[(b)] The group $\Comm(G)_S$ is Hausdorff if and only if $\VZ(G)$ is closed.
If these conditions hold, $\Comm(G)_S$ is a t.d.l.c. group.
\item[(c)] If $\VZ(G)=\{1\}$, then $\VZ(\Comm(G)_S)=\{1\}$ as well.
\item[(d)] Assume that $\VZ(G)=\{1\}$. Then $\Comm(G)_S$ is unimodular
if and only if any two isomorphic open subgroups of $G$
are of the same index.
\item[(e)]
Assume that $G$ is finitely generated and $\VZ(G)=\{1\}$.
Then $\Comm(G)_S$ is uniscalar
if and only if for every virtual automorphism $\phi\colon U\to V$
of $G$ there is an open subgroup $W\subseteq U$ such that $\phi(W)=W$.
\end{itemize}
\end{prop}

\begin{proof}
Parts (a) and (b) are straightforward, so we only prove (c), (d) and (e).

(c) Let $g\in \VZ(\Comm(G)_S)$, and let $V$ be an open subgroup
of $G$, such that $[g,\iota(V)]=\{1\}$. Let $\phi\in\VAut(G)$ be a virtual
automorphism representing $g$, and let $U$ be an open subgroup on which
$\phi$ is defined. The equality $[g,\iota(V)]=\{1\}$ implies that
$x^{-1}\phi(x)\in \VZ(G)$ for every $x\in U\cap V$. Since $\VZ(G)=\{1\}$,
we conclude that $\phi$ acts trivially on $U\cap V$,
whence $g=[\phi]=1$.

(d) Take any $g=[\phi]\in\Comm(G)$.
Let $U$ be an open subgroup of $G$ on which $\phi$ is defined, and let $V=\phi(U)$.
Let $\mu$ be a fixed Haar measure on $\Comm(G)_S$,
and let $\Delta\colon\Comm(G)\to\RR$
denote the modular function of $\Comm(G)_S$.
Then
\begin{equation}
\label{eq:mod}
\Delta(g)=\mu(V)/\mu(U)=|G:U|/|G:V|,
\end{equation}
which immediately implies the assertion of part (d).

(e) A t.d.l.c. group is uniscalar if and only if every element normalizes some
open compact group. Thus the `if' part is obvious. Now assume that $\Comm(G)_S$
is uniscalar. Given $\phi\in \VAut(G)$, let $Y$ be an open compact subgroup
of $\Comm(G)_S$ normalized by $[\phi]$. Since $G$ is finitely generated,
so is $Y$, and thus there exists a characteristic subgroup $Y^\prime$ of $Y$
which is contained in $\image(\iota_U)$.
Hence $\phi(W)=W$ for $W=\iota_U^{-1}(Y^\prime)$.
\end{proof}

In addition to having a transparent structure, the strong topology does have
practical applications. In the next subsection we will see that the group
$\Comm(G)_S$ can be thought of as the universal envelope of $G$, provided $\VZ(G)=\{1\}$.
However, as the following examples show, the strong
topology does not have to coincide with the ``natural'' topology on $\Comm(G)$.

\begin{example}
\label{ex:strong_defect}
\rm (a) Let $G=\Z_p$. Then $\Comm(\Z_p)$ is clearly isomorphic to
$\Q_p^{*}$ as an abstract group, but $\ca{T}_S$ is the discrete topology.

(b) Let $G=\SL_n(\Fpt)$. Then $\Comm(G)$ is isomorphic to a finite extension
of $\PGL_n(\RFpt)\rtimes \Aut(\Fpt)$ and carries a natural topology induced from the local field $\RFpt$.
The subgroup $\PGL_n(\RFpt)$ of $\PGL_n(\RFpt)\rtimes \Aut(\Fpt)$
is open with respect to the strong topology, but not with respect to the local field
topology.
\end{example}

The deficiencies of the strong topology on $\Comm(G)$ illustrated by this example
are due to the fact that
the maps  $\rho_U\colon\Aut(U)\to \Comm(G)_S$, with $U$ open in $G$, are not necessarily continuous
with respect to the standard topology on $\Aut(U)$.
The strongest topology on $\Comm(G)$ which makes all these maps continuous and turns
$\Comm(G)$ into a topological group will be introduced in Section~\ref{ss:topaut}.
This topology will be called the $\Aut$-topology.

\subsection{Commensurators as universal envelopes}
\label{ss:unienv}
Let $G$ be a profinite group. In the introduction we defined an envelope of $G$
to be any group $L$ which contains $G$ as an open subgroup. For various purposes it
will be convenient to think of envelopes in a more categorical way:
\begin{defi}
Let $G$ be a profinite group. An {\it envelope} of $G$ is a pair $(L,\eta)$
consisting of a topological group $L$ and an injective homomorphism $\eta\colon G\to L$
such that $\eta(G)$ is open in $L$, and $\eta$ maps $G$ homeomorphically onto $\eta(G)$.
The group $L$ itself will also be referred to as an envelope of $G$ whenever the
reference to the map $\eta$ is inessential.
\end{defi}
The next proposition shows
that $\Comm(G)_S$ can be considered as a universal open envelope of $G$,
provided $G$ has trivial virtual center.

\begin{prop}
\label{prop:unienv}
Let $G$ be a profinite group, and let
$(L,\eta)$ be an envelope for $G$.
Then there exists a canonical continuous open
homomorphism $\eta_\ast\colon L\to\Comm(G)_S$ making the following diagram commutative:
\begin{equation}
\label{cd:unienv}
\xymatrix{
G\ar[r]^{\eta}\ar[rd]^{\iota_G}&L\ar@{-->}[d]^{\eta_\ast}\\
&\Comm(G)_S
}
\end{equation}
The kernel of $\eta_{\ast}$ is equal to $\VZ(L)$. Furthermore,
if $\VZ(G)=\{1\}$, then $\eta_\ast$ is the unique map making
\eqref{cd:unienv} commutative.
\end{prop}

\begin{proof}
For simplicity of notation, we shall identify $G$ with $\eta(G)$.
Let $l\in L$, and let $U = G\, \cap\, l^{-1}G l$.
Then $U$ and $l U l^{-1}$ are open subgroups of $G$, and left conjugation by $l$ induces a virtual isomorphism
$i_l^\prime\colon U\to l U l^{-1}$ of $G$.
It is straightforward to check that the induced map $\eta_\ast\colon L\to\Comm(G)_S$
given by $\eta_\ast(l)=[i_l^\prime]$ has the desired properties.

Assume that $\VZ(G)=\{1\}$. To prove uniqueness, assume that there is another map
$j\colon L\to \Comm(G)$ making
the above diagram commutative. Then $j(x)=\eta_*(x)$ for every $x\in G$.
Now take any $l\in L$, and let $V$ be an open subgroup
of $G$ such that $l V l^{-1}\subset G$. Then for every $x\in V$ we have
$j(l xl^{-1})=\eta_*(lxl^{-1})$, whence
$j(l) j(x) j(l)^{-1}=\eta_*(l)j(x)\eta_*(l)^{-1}$. Thus, if $h=j(l)^{-1}\eta_*(l)$,
then $h$ centralizes $j(V)$, which is an open subgroup of $\Comm(G)_S$. Since
$\Comm(G)_S$ has trivial virtual center by Proposition~\ref{prop:strong}(c),
we conclude that $h=1$, whence $j(l)=\eta_*(l)$.
\end{proof}

Next we turn to the following question: which t.d.l.c. groups
arise as commensurators of profinite groups with trivial virtual center.
First, observe that if $G$ is a profinite group, the canonical map
$\iota\colon G\to\Comm(G)_S$
is an isomorphism if and only if
\begin{itemize}
\item[(i)]$\VZ(G)=\{1\}$ and
\item[(ii)] Any virtual automorphism of $G$ is given by conjugation by some $g\in G$.
\end{itemize}

A group $G$ (not necessarily profinite) satisfying (i) and (ii) will be called hyperrigid.
It will be convenient to reformulate the definition of hyperrigidity as follows:

\begin{defi} A topological group $L$ is called {\it hyperrigid} if for every
topological isomorphism $\phi:U\to V$ between open compact subgroups $U$ and $V$ of $L$
there exists a unique element $g_{\phi}\in L$ such that $\phi(x)=g_{\phi}\,x\,{g_{\phi}}^{-1}$
for every $x\in U$.
\end{defi}

The following proposition shows that hyperrigidity is a built-in and defining property of
commensurators of profinite groups with trivial virtual center.

\begin{prop}
\label{prop:rigid}
Let $G$ be a profinite group with trivial virtual center. Then $\Comm(G)_S$ is a hyperrigid t.d.l.c. group.
Moreover, $(\Comm(G)_S,\iota_G)$ is the unique (up to isomorphism) hyperrigid envelope of $G$.
\end{prop}

\begin{proof}
Since $\VZ(G)=\{1\}$, we can identify $G$ with $\iota_G(G)$.
Let $\phi\colon U\to V$ be an isomorphism of open compact subgroups of
$\Comm(G)_S$. Note that the groups $U$ and $G$ are commensurable, so
$\Comm(G)_S$ can be canonically identified with $\Comm(U)_S$.
Let $\phi^\prime=\phi\vert_{U^\prime}\colon U^\prime\to V^\prime$, where
$U^\prime=\phi^{-1}(U\cap V)$ and $V^\prime=U\cap V$. Then $\phi'$
is a virtual automorphism of $U$, and let $g=[\phi^\prime]\in\Comm(U)$.
The isomorphism $\phi\circ i_{g^{-1}}\colon gUg^{-1}\to V$ restricted to
$W=gU' g^{-1}$ is equal to $\iid_W$.
By Proposition~\ref{prop:isot}, $V=gUg^{-1}$ and $\phi=i_g\vert_U$, so
$\phi(x)=gxg^{-1}$ for every $x\in U$.
The uniqueness of $g$ with this property is clear since $\VZ(\Comm(U)_S)=\{1\}$.
This shows that $\Comm(G)_S=\Comm(U)_S$ is hyperrigid,
and $(\Comm(G)_S,\iota_G)$ is a hyperrigid envelope of $G$.
It remains to show that it is unique up to isomorphism.
Suppose that $(L,\eta)$ is a hyperrigid envelope of $G$.
Hyperrigidity of $L$ yields a map $\beta\colon \VAut(G)\to L$ given by
$\beta(\phi)=g_{\phi}$, which defines a group homomorphism $\beta_\ast\colon\Comm(G)\to L$.
A straightforward computation shows that $\beta_\ast$ is the inverse of
$\eta_\ast$ where $\eta_{\ast}\colon L\to\Comm(G)_S$ is the canonical map
defined in Proposition \ref{prop:unienv}. Since $\eta_{\ast}$ is continuous
and open, we conclude that $L\cong \Comm(G)_S$.
\end{proof}
\begin{cor}
A t.d.l.c. group $L$ is hyperrigid if and only if $L\cong \Comm(G)_S$
for some profinite group $G$ with trivial virtual center.
\end{cor}

\subsection{Rigid envelopes and inner commensurators}
\label{ss:inner}
Let $G$ be a profinite group with trivial virtual center.
Given any envelope $(L,\eta)$ of $G$, one can always
consider the ``larger'' envelope $(\Aut(L)_S, i_L\circ\eta)$
where $i_L\colon L\to\Aut(L)_S$ is the canonical map defined in \S\ref{ss:autotop}.
By Proposition~\ref{prop:unienv}, we have a canonical map
\begin{equation}
\label{eq:kappa}
\kappa_{L,G}=(i_L\circ\eta)_*\colon \Aut(L)_S\to\Comm(G)_S
\end{equation}
which makes the following diagram commutative:
\begin{equation}
\xymatrix{
L\ar[d]_{i_L}&G\ar[l]_\eta \ar[d]^{\iota_G}\\
\Aut(L)\ar[r]^{\kappa_{L,G}}&\Comm(G)
}
\end{equation}

It is clear that $\kernel(\kappa_{L,G})=\TAut(L)$.
The question of when $\kappa_{L,G}$ is an isomorphism naturally leads to the
notion of a rigid group.

\begin{defi} A t.d.l.c. group $L$ will be called {\it rigid}, if for every topological isomorphism
$\phi\colon U\to V$ of open compact subgroups $U$ and $V$ of $L$ there exists a unique automorphism
$\phi_\circ\in \Aut(L)$ such that the following diagram commutes.
\begin{equation}
\label{dia:rigid2}
\xymatrix{
U\ar[r]^\phi\ar[d]_{}&V\ar[d]^{}\\
L\ar[r]^{\phi_\circ}&L
}
\end{equation}
\end{defi}
\begin{rem} One can think of rigid groups as the groups satisfying the analogue of
Mostow's strong rigidity theorem with open compact subgroups
playing the role of lattices.
\end{rem}

It is easy to see that hyperrigidity implies rigidity. Indeed, If $L$ is hyperrigid, there exists an (inner)
automorphism $\phi_\circ$ that makes \eqref{dia:rigid2} commutative. Furthermore, $\VZ(L)=\{1\}$, and
hence $\TAut(L)=\{1\}$ by Proposition~\ref{prop:jiatriv}. This yields the uniqueness of $\phi_{\circ}$
in \eqref{dia:rigid2}.

The following proposition shows the importance of rigid envelopes
for the computation of commensurators.

\begin{prop}
\label{prop:autcomm}
Let $L$ be a t.d.l.c. group with $\VZ(L)=\{1\}$, let $G$ be an open compact subgroup of $L$,
and let $\eta: G\to L$ be the inclusion map. The following conditions are equivalent:
\begin{itemize}
\item[(a)] $L$ is rigid.
\item[(b)] $\kappa_{L,G}:\Aut(L)_S\to\Comm(G)_S$ is an isomorphism
\item[(c)] $\eta_*(L)$ is a normal subgroup of $\Comm(G)$.
\end{itemize}
\end{prop}

\begin{proof}
(a) $\Rightarrow$ (b) If $L$ is rigid, the map $\beta\colon\VAut(G)\to\Aut(L)$ given by $\beta(\phi)=\phi_\circ$
(where $\phi_{\circ}$ is as in \eqref{dia:rigid2}),
induces a homomorphism  $\beta_*\colon\Comm(G)\to \Aut(L)$.
A straightforward computation shows that $\beta_{\ast}$ is the inverse of $\kappa_{L,G}$,
so $\kappa_{L,G}$ is an isomorphism.

(b) $\Rightarrow$ (c) This is clear since $\eta_*(L)=\kappa_{L,G}(i_L(L))$
and $i_L(L)=\Inn(L)$ is normal in $\Aut(L)$.

(c) $\Rightarrow$ (a) Let $L'=\eta_*(L)$. Since $L'$ is isomorphic to $L$, it is enough
to prove that $L'$ is rigid. Let $\phi:U\to V$ be an isomorphism between open compact
subgroups of $L'$. Since $L'$ is open in $\Comm(G)$ and $\Comm(G)$ is hyperrigid,
there exists $g\in \Comm(G)$ such that $\phi(x)=gxg^{-1}$ for every $x\in U$.
But $L'$ is normal in $\Comm(G)$, and thus $\phi$ extends to the automorphism $\phi_{\circ}=i_g$
of $L'$. Furthermore, this extension is unique since $\TAut(L')=\{1\}$. The latter
follows from Proposition~\ref{prop:jiatriv} since $L'\cong L$ and hence $\VZ(L')=\{1\}$.
Therefore, $L'$ is rigid.
\end{proof}

Let $G$ be a profinite group with trivial virtual center. The equivalence of (a) and (b) in
Proposition~\ref{prop:autcomm} shows that whenever we find a rigid envelope $L$ of $G$,
the commensurator $\Comm(G)$ can be recovered from $L$. On the other hand,
it is natural to ask whether $L$ can be
recovered from $G$. If $L$ is topologically simple, this question is answered in the positive
using the notion of inner commensurator (see Corollary~\ref{cor:autcomm}). By Theorem~\ref{Chevalley}
below, this result applies for instance when $L=\PSL_n(F)$, $G=\PSL_n(O)$, where
$F$ is a local field and $O$ is its ring of integers.

\begin{defi}
{\rm Let $G$ be a profinite group. The {\it inner commensurator} of $G$
is the normal subgroup of $\Comm(G)$ generated by $\iota(G)$.
It will be denoted by $\IC (G)$.}
\end{defi}
\begin{rem}It is easy to see that the inner commensurator $\IC (G)$ is not a function
of the commensurability class of $G$. In fact, it may happen that $G$ is a finite
index subgroup of $H$, but the index $[\IC(H):\IC(G)]$ is uncountable.
For instance, let $G=\SL_n(\Fpt)$ and
$H=\SL_n(\Fpt)\rtimes \la f\ra$, where $f$ is an automorphism of $\Fpt$ of finite order.
Then by Theorem~\ref{Chevalley} we have $\IC(G)\cong \PSL_n(\Fpt)$ while
$\IC(H)\cong \PSL_n(\Fpt)\rtimes U$ where $U$ is an open subgroup of $\Aut(\Fpt)$.
\end{rem}

Once again, assume that $\VZ(G)=\{1\}$. By Proposition~\ref{prop:autcomm}, $\IC(G)$
is a rigid envelope of $G$. In fact, $\IC(G)$ is the smallest rigid envelope of $G$,
that is, if $(L,\eta)$ is any rigid envelope of $G$, then $\IC(G)\subseteq \eta_*(L)$.
On the other hand, if $(L,\eta)$ is any topologically simple envelope of $G$, then
the reverse inclusion holds: $\eta_*(L)\subseteq \IC(G)$ (see Proposition~\ref{tdlc:IC}(b)).
Combining these observations, we obtain the following useful fact:
\begin{cor}
\label{cor:autcomm}
Let $G$ be a profinite group with $\VZ(G)=\{1\}$, and let $(L,\eta)$ be a topologically simple
envelope of $G$. Then $L$ is rigid if and only if $\eta_*(L)=\IC(G)$.
\end{cor}

\subsection{Commensurators of algebraic groups}
In this subsection we explicitly describe commensurators for two
important classes of profinite groups: open compact subgroups of
simple algebraic groups over local fields and compact $p$-adic analytic groups.

The following theorem appears as Corollary~0.3 in \cite{pink}:

\begin{thm}[Pink]
\label{thm:Pink} Let $F,F'$ be local fields, let $\dbG$ (resp. $\dbG'$)
be an absolutely simple simply connected algebraic group over $F$ (resp. $F'$),
and let $G$ (resp. $G'$) be an open compact subgroup of $\dbG(F)$ (resp. $\dbG'(F')$).
Let $\phi: G\to G'$ be an isomorphism of topological groups. Then there exists a unique isomorphism of algebraic groups
$\dbG\to \dbG'$ over a unique isomorphism of local fields $F\to F'$ such that
the induced isomorphism $\dbG(F)\to \dbG'(F')$ extends $\phi$.
\end{thm}

In the special case when $\dbG'=\dbG$ and $F'=F$ is non-archimedean,
Theorem~\ref{thm:Pink} can be considered as a combination
of two results. First, it implies that the t.d.l.c. group $\dbG(F)$ is rigid,
and thus $\Comm(G)$ is canonically isomorphic to $\Aut(\dbG(F))$. Second, Theorem~\ref{thm:Pink}
shows that $\Aut(\dbG(F))$ is naturally isomorphic to $(\Aut\dbG)(F)\rtimes \Aut(F)$
where $\Aut\dbG$ is the group of automorphisms of the algebraic group $\dbG$.
Note that when $\dbG$ is isotropic over $F$, the last result is a special case of Borel-Tits'
theorem~\cite{boti:aut} (which applies to simple algebraic groups over arbitrary fields).
Thus, we obtain the following theorem:

\begin{thm}
Let $F$ be a non-archimedean local field, let
$\dbG$ be an absolutely simple simply connected algebraic group over $F$,
and $G$ an open compact subgroup of $\dbG(F)$. Then
$\Comm(G)$ is canonically isomorphic to $(\Aut\dbG)(F)\rtimes \Aut(F)$.
In particular, if $\dbG$ is split over $F$, then
$\Comm(G)\cong \dbG_{ad}(F)\rtimes (X\times \Aut(F))$
where $\dbG_{ad}$ is the adjoint group of $\dbG$, $X$ is the finite group
of automorphisms of the Dynkin diagram of $\dbG$,
and $\Aut(F)$ is the group of field automorphisms of $F$.
\label{Chevalley}
\end{thm}

Commensurators of $p$-adic analytic groups can be computed as follows:

\begin{thm}
\label{thm:padic}
Let $G$ be a compact $p$-adic analytic group. Then one has
a canonical isomorphism
\begin{equation}
\Comm(G)\simeq\Aut_{\Q_p}(\eu{L}(G)).
\end{equation}
where $\eu{L}(G)$ is the $\Q_p$-Lie algebra of $G$ as introduced by Lazard.
\end{thm}

A statement very similar to Theorem~\ref{thm:padic} appears in Serre's book on Lie algebras and Lie groups~\cite{serre:lalg},
and one can easily deduce Theorem~\ref{thm:padic}
from that statement using elementary theory of $p$-adic analytic groups.
For completeness, we shall give a slightly different proof
of Theorem~\ref{thm:padic} in the Appendix.

\section{Topologically simple totally disconnected locally compact groups}
\label{s:tdlc}
Totally disconnected locally compact (t.d.l.c) groups have been
a subject of increasing interest in recent years, starting with
a seminal paper of Willis~\cite{george:scale}. In this section we use commensurators
to study the structure of open subgroups of topologically simple t.d.l.c. groups.
More specifically, we shall address the following three problems.

\begin{itemize}
\item[(1)] Given a profinite group $G$, describe all topologically simple envelopes
of $G$.

\item[(2)] Given a subclass $\grL$ of topologically simple t.d.l.c. groups,
find restrictions on the structure of profinite groups which have at least one
envelope in $\grL$.

\item[(3)] Construct new examples of topologically simple t.d.l.c. groups.
\end{itemize}

If $\grL$ consists of all topologically simple t.d.l.c. groups, Problem~(2) is unlikely
to have a satisfactory answer, as suggested by a recent paper of Willis \cite{george:sim}.
In \cite{george:sim}, Willis constructed a topologically simple
t.d.l.c. groups containing an open compact abelian subgroup.
Note that a group with such a property must coincide with its virtual center.
In the same paper, it was shown that a topologically simple t.d.l.c. group $L$
cannot have this or other similar ``pathological'' properties, provided $L$ is compactly generated.
We will address Problem~(2) for the class $\grL$
consisting of compactly generated topologically simple t.d.l.c. groups.

\subsection{The open-normal core}
For a t.d.l.c. group $L$, let $\ON(L)$ denote the set of all open normal subgroups of $L$.
As $L\in\ON(L)$, this set is non-empty. We define the {\it open-normal core} of $L$ by
\begin{equation}
\label{eq:defophea}
\Ht(L)=\bigcap_{N\in\ON(L)} N.
\end{equation}
Hence $\Ht(L)$ is a closed characteristic subgroup of $L$ which is contained in every open normal subgroup of $L$,
and $\Ht(L)$ is maximal with respect to this property.
Note that $\Ht(L)=\{1\}$ if and only if $L$ is pro-discrete, that is,
$L$ is an inverse limit of discrete groups.

If $L$ is a t.d.l.c. group, then so is $\Ht(L)$, and we can consider the ``core series''
$\{\Ht^k(L)\}_{k=0}^{\infty}$, defined by $\Ht^0(L)= L$ and $\Ht^{k+1}(L)\deq\Ht(\Ht^k(L))$ for $k\geq 1$.
This is a descending series consisting of closed characteristic subgroups of $L$, and
it may be even extended transfinitely.
\vskip .1cm
If $L=\Comm(G)_S$ for some profinite group $G$, the open-normal core $\Ht(L)$
has the following alternative description.
\begin{prop}
Let $G$ be a profinite group. Then $\Ht(\Comm(G)_S)$ is equal to the intersection
$\bigcap\limits_{U\in\mathcal U_G}\IC(U)$ where $\mathcal U_G$ is the set of all open subgroups of $G$.
\end{prop}
\begin{proof}
Any open subgroup $N$ of $\Comm(G)_S$ contains $\iota(U)$ for some $U\in\mathcal U_G$.
If $N$ is also normal, then $N$ contains $\IC(U)$, so
$\bigcap\limits_{U\in\mathcal U_G}\IC(U)\subseteq\Ht(\Comm(G)_S)$.
The reverse inclusion is obvious.
\end{proof}

\subsection{Existence and uniqueness of topologically simple envelopes}
\label{ss:uniqueenv}
Let $G$ be a profinite group, and let $(L,\eta)$ be an envelope of $G$.
By Proposition~\ref{prop:unienv}, there exists a canonical (continuous) map
$\eta_\ast\colon L\to\Comm(G)_S$.
If $L$ is topologically simple, the following stronger statement holds:

\begin{prop}
\label{tdlc:IC}
Let $G$ be a profinite group, and let $(L,\eta)$ be a
topologically simple envelope of $G$. The following hold:
\begin{itemize}
\item[(a)] If $\VZ(G)=\{1\}$, then $\VZ(L)=\{1\}$ (and therefore, $\eta_{*}$ is injective).

\item[(b)] $\eta_\ast(L)\subseteq \Ht(\Comm(G)_S)$.
\end{itemize}
\end{prop}
\begin{proof} (a) Suppose that $\VZ(L)\neq\{1\}$. Then $\VZ(L)$ is dense in $L$. Since $\eta(G)$
is open in $L$, we have $\VZ(\eta(G))=\VZ(L)\cap\eta(G)$. Thus,
$\eta(\VZ(G))=\VZ(\eta(G))$ is dense in $\eta(G)$, contrary to the assumption $\VZ(G)=\{1\}$.

(b) If $N$ is any open normal subgroup of $\Comm(G)_S$, then $N\cap \eta_*(L)$
is an open normal subgroup of $\eta_*(L)$. Since $\eta_*(L)$ is topologically simple,
$\eta_*(L)$ must be contained in $N$.
\end{proof}

Note that if $\VZ(G)=\{1\}$ and $G$ has at least one topologically simple envelope,
then $\Ht(\Comm(G)_S)$ is an envelope of $G$ by Proposition~\ref{tdlc:IC}(b).
If in addition $\Ht(\Comm(G)_S)$ happens to be a topologically simple group, it
is natural to expect that $\Ht(\Comm(G)_S)$ is the unique (up to isomorphism)
topologically simple envelope of $G$.
We do not know if such a statement is true in general, but we can prove it for the class
of algebraic groups covered by Theorem~\ref{Chevalley}:

\begin{prop}
\label{prop:unique}
Let $\dbG$ be a connected, simply connected simple algebraic group defined
and isotropic over a non-archimedean local field $F$.
Let $G$ be an open compact subgroup of $\dbG(F)$, with $\VZ(G)=\{1\}$. Then
$\Ht(\Comm(G)_S)$ is topologically simple. Furthermore, if $G$ has at least one
topologically simple envelope $(L,\eta)$, then $\eta_*(L)=\Ht(\Comm(G)_S)$.
\end{prop}

\begin{proof}
It is an easy consequence of Theorem~\ref{Chevalley} that
$\Ht(\Comm(G)_S)$ is isomorphic to the group $S=\dbG(F)/\Zen(\dbG(F))$.
By Tits' theorem~\cite[I.2.3.2(a)]{Margulis:book}, $S$ is simple (even as an abstract group).

Now let $(L,\eta)$ be a topologically simple envelope of $G$. By Proposition~\ref{tdlc:IC},
we can identify $L$ with a subgroup of $S$. Since $\dbG$ is isotropic and simply connected,
$S$ is generated by unipotent elements (see \cite[I.2.3.1(a)]{Margulis:book}),
and therefore by a theorem of Tits \cite{prasad:bull}, every open subgroup of $S$
is either compact or equals the entire group $S$. Since $L$ is topologically simple
and infinite, it cannot be compact. Thus, $L=S$.
\end{proof}

Proposition~\ref{tdlc:IC} yields the first obstruction to the existence of
a topologically simple envelope:
\begin{prop}
\label{prop:prodisc}
Let $G$ be a profinite group with $\VZ(G)=\{1\}$, and suppose that
$\Ht(\Comm(G)_S)=\{1\}$. Then $G$ does not have a topologically simple envelope.
\end{prop}

\begin{cor}
\label{cor:prodisc}
Let $G$ be a profinite group with $\VZ(G)=\{1\}$. Suppose that
$G$ has an open subgroup $U$ such that
\begin{itemize}
\item[(i)] $\Comm(G)=\Autbar(U)$
\item[(ii)] $U$ has a base $\ca{C}$ of neighborhoods of the identity consisting of characteristic subgroups.
\end{itemize}
Then $\Ht(\Comm(G)_S)=\{1\}$, and hence $G$ does not have a topologically simple envelope.
\end{cor}
\begin{proof} For every $V\in\ca{C}$ the group $\iota(V)$ is open in $\Comm(G)_S$. Furthermore,
$\iota(V)$ is normal in $\Autbar(U)$ since $V$ is characteristic in $U$. Since
$\bigcap_{V\in\ca{C}}\iota(V)=\{1\}$, we conclude that $\Ht(\Comm(G)_S)=\Ht(\Autbar(U)_S)=\{1\}$,
so we are done by Proposition~\ref{prop:prodisc}.
\end{proof}

\begin{rem} A profinite group $G$ such that $\VZ(G)=\{1\}$ and $\Ht(\Comm(G)_S)=\{1\}$
will be said to have {\it pro-discrete type} -- this condition arises naturally in our
classification of hereditarily just-infinite profinite groups (see Section~\ref{s:hji}).
We do not know any examples of groups of pro-discrete type not satisfying the hypothesis
of Corollary~\ref{cor:prodisc}.
\end{rem}

Corollary~\ref{cor:prodisc} has interesting consequences (see Section~\ref{s:hji}),
but it can only be applied to groups whose commensurators are known. In order to obtain
deeper results on the structure of open compact subgroups of topologically simple
t.d.l.c. groups, we now restrict our attention to compactly generated groups.

\subsection{Compactly generated, topologically simple envelopes}
\label{ss:cgtsenv}
We begin with two general structural properties of
compactly generated topologically simple t.d.l.c. groups.

\begin{prop}
\label{prop:WilJA}
Let $L$ be a compactly generated topologically simple t.d.l.c. group.
Then $L$ is countably based.
\end{prop}

\begin{proof} Since $L$ is compactly generated, it is obviously $\sigma$-compact,
that is, a countable union of compact subsets. By \cite[Prop.~2.1]{george:sim},
a $\sigma$-compact topologically simple t.d.l.c. group is metrizable and
thus countably based.
\end{proof}

\begin{thm}
\label{cor:simvir}
Let $L$ be a non-discrete, compactly generated topologically simple t.d.l.c. group.
Then $\VZ(L)=\{1\}$, and therefore $\VZ(G)=\{1\}$ for every open compact subgroup $G$ of $L$.
\end{thm}

\begin{proof} Fix an open compact subgroup $U$ of $L$, and suppose that $\VZ(L)\neq \{1\}$.
Then $\VZ(L)$ is dense and normal in $L$, so $L=U\cdot\VZ(L)$.
Since $L$ is compactly generated, there exist elements
$x_1,\ldots x_r\in L$ such that
\begin{equation}
\label{eq:tsim1}
\langle x_1,\ldots, x_r, U\rangle=L.
\end{equation}
As $L=U\cdot\VZ(L)$, we may assume that
$x_1,\ldots x_r\in \VZ(L)$.

Let $V=\bigcap_{i=1}^r \Cent_L(x_i)$. Then $V$ is open in $L$, and thus
$V\cap U$ is compact and open in $L$. Hence $V\cap U$ contains a
subgroup $N$ which is normal in $U$ and open in $L$. By construction, $N$
is normalized by $U$ and each $x_i$. Thus (\ref{eq:tsim1}) implies
that $N$ is normal in $L$. Since $L$ is topologically simple, it follows
that $N=L$, which is impossible as $L$ is not compact.
\end{proof}

\begin{defi}
Let $G$ be a profinite group with trivial virtual center.
A closed subgroup $N$ of $G$ will be called {\it sticky}, if
$N\cap gNg^{-1}$ has finite index in $N$ and $gNg^{-1}$ for all $g\in\Comm(G)$.
\end{defi}

If $N$ is a sticky subgroup of $G$, we have a natural homomorphism of (abstract) groups
$s_N\colon \Comm(G)\to\Comm(N)$ given by
\begin{equation}
\label{eq:sticky1}
s_N(g)=[i_g\vert_{N\cap gNg^{-1}}],
\end{equation}
where $i_g\colon\Comm(G)\to\Comm(G)$ denotes left conjugation by $g\in\Comm(G)$.

The following theorem shows that if $G$ is a profinite group with a compactly
generated topologically simple envelope, then $G$ does not have
sticky subgroups of a certain kind.

\begin{thm}
\label{thm:rest}
Let $G$ be a profinite group containing a non-trivial closed normal subgroup $N$
with the following properties:
\begin{itemize}
\item[(i)] $N$ is sticky in $G$.
\item[(ii)] $N$ has non-trivial centralizer in $G$.
\end{itemize}
Then $G$ does not have a compactly generated, topologically simple envelope.
\end{thm}

\begin{proof} If $N$ is a finite normal subgroup, then $\Cent_G(N)$ must be open.
Thus, $\VZ(G)$ is non-trivial, so $G$ does not have
a compactly generated topologically simple envelope by Theorem~\ref{cor:simvir}.

Now assume that $N$ is infinite, and
suppose that $G$ has a compactly generated topologically simple envelope
$(L,\eta)$.
By Proposition~\ref{prop:WilJA}, $L$ is countably based. Then $G$
is also countably based, and furthermore, $G$ has a base $G_1\supset G_2\supset\ldots$
of neighborhoods of the identity where each $G_i$ is normal in $G$. For each $i\in\mathbb N$ we set $N_i=N\cap G_i$.
Then each $N_i$ is normal in $G$, and $\{N_i\}$ is a base of neighborhoods of the identity for $N$.

Let $\eta_*:L\to \Comm(G)_S$ be the canonical map. Note that $\eta_*$ is injective by Theorem~\ref{cor:simvir}.
Condition (i) yields a homomorphism $s_N\colon\Comm(G) \to \Comm(N)$ (which is not necessarily continuous).
Let $\alpha= s_N\circ\eta_\ast$. Condition (ii) implies that $\alpha$ restricted to $G$ is not
injective. Let $K=\kernel(\alpha)$.
Then for every $x\in K$ there exists an open subgroup $U_x$ of $N$ such that $x$ centralizes $U_x$.
Therefore,
\begin{equation}
\label{eq:thmloc1}
K\subseteq \bigcup _{U \leq_o N} \Cent_L(U)= \bigcup_{i\geq 1}\Cent_L(N_i).
\end{equation}
Since $K\neq \{1\}$ and $L$ is topologically simple, $K$ must be dense in $L$.
Hence
\begin{equation}
\label{eq:thmloc2}
L=K G\subseteq \bigcup_{i\geq 1} \Cent_L(N_i) G.
\end{equation}
As $L$ is compactly generated, there exist finitely many elements $x_1,\ldots,x_n\in L$
such that $L=\langle\, x_1,\ldots, x_n,G\,\rangle$. By \eqref{eq:thmloc2}, we may assume
that $x_1,\ldots,x_n\in\Cent_L(N_m)$ for some $m\in \N$.
Hence $N_m$ is a non-trivial normal subgroup
of $L$, which contradicts topological simplicity of $L$.
\end{proof}

It is quite possible that the existence of a compactly generated topologically simple envelope for a profinite group $G$
yields much stronger restrictions on sticky subgroups of $G$ than the ones given by Theorem~\ref{thm:rest}.
We do not even know the answer to the following basic question:
\begin{quest}
\label{qu:sticky}
Let $G$ be a profinite group with a compactly generated
topologically simple envelope.
Is it true that every infinite sticky subgroup of $G$ is open?
\end{quest}
We shall now discuss various applications of Theorem~\ref{thm:rest}.
\begin{cor}
\label{cor:stupid}
Let $G$ be an infinite pro-(finite nilpotent) group, which has a compactly generated
topologically simple envelope. Then $G$ is a pro-$p$ group for some prime number $p$.
\end{cor}

\begin{proof} By hypothesis $G$ is the cartesian product of its pro-$p$ Sylow subgroups
$\{G_p\}$. By Theorem~\ref{cor:simvir}, $G$ has trivial virtual center, so each
$G_p$ is either trivial or infinite.

Assume that there exist two distinct prime numbers $p$ and $q$ such that
$G_p$ and $G_q$ are non-trivial, and thus infinite. Let $N=G_p$. Then
$N$ has non-trivial centralizer since $N$ commutes with $G_q$. Moreover,
$N$ is sticky in $G$ since every closed subgroup of $N$ is pro-$p$, and
conversely, every pro-$p$ subgroup of $G$ is contained in $N$. Thus,
$N$ satisfies all conditions of Theorem~\ref{thm:rest}, which
contradicts the existence of a compactly generated topologically simple envelope for $G$.
\end{proof}

Here is another important case where the hypotheses of Theorem~\ref{thm:rest} are satisfied.
Recall that the {\it Fitting subgroup} of a group $G$ is the subgroup generated by all normal nilpotent subgroups of $G$.

\begin{prop}
\label{prop:sticky}
Let $G$ be a profinite group with trivial virtual center,
and suppose that the Fitting subgroup $R_G$ of $G$ is nilpotent.
Then $R_G$ is sticky.
\end{prop}

\begin{rem}
Note that by Fitting's theorem, $R_G$ is nilpotent if and only if $G$
contains a maximal nilpotent normal subgroup (such subgroup is automatically
closed and unique). It is known that $R_G$ is nilpotent for every
linear group $G$ (see \cite[\S 8.2.ii]{wehr:book}).
\end{rem}

\begin{proof}
For simplicity we identify $G$ with the open compact subgroup
$\iota_G(G)$ of $\Comm(G)_S$. Let $g\in\Comm(G)_S$ and put
\begin{equation}
\label{eq:stpf1}
H = i_g(G)=gGg^{-1},\qquad R_H = i_g(R_G).
\end{equation}
Let $X$ be an open normal subgroup of $G$ contained in $G\cap H$. Note that
$|HG/X|<\infty$.
The subgroup $R_H\cap X$ is a closed normal nilpotent subgroup of $X$.
Let $\ca{R}$ be a set of coset representatives of
$G/X$. Hence by Fitting's theorem
\begin{equation}
\label{eq:stpf3}
Y\deq \textstyle{\prod_{r\in\ca{R}} r^{-1}(R_H\cap X)}r
\end{equation}
is a closed normal nilpotent subgroup of $G$, and thus contained in $R_G$.
Therefore, $R_H\cap X\subseteq R_H\cap Y\subseteq R_H\cap R_G$, and so
\begin{equation}
\label{eq:stpf4}
|R_H / R_H\cap R_G|\leq |R_H/R_H\cap X|=|R_H X/X|\leq |H G/X|<\infty,
\end{equation}
Thus, $R_H\cap R_G$ is of finite index in $R_H$. Changing the roles of $H$ and $G$
then yields the claim.
\end{proof}

The following result is a direct consequence of
Proposition \ref{prop:sticky} and Theorem~\ref{thm:rest}.

\begin{cor}
\label{cor:prestupid}
Let $G$ be a profinite group with trivial virtual center,
and suppose the Fitting subgroup of $G$ is nilpotent and non-trivial.
Then $G$ does not have a compactly generated topologically simple envelope.
\end{cor}

Some non-trivial examples where Corollary~\ref{cor:prestupid} is applicable
are collected in the following statement:

\begin{cor}
\label{cor:nilpotent} Assume that one of the following holds:
\begin{itemize}
\item[(a)] $G$ is an open compact subgroup of $\dbG(F)$, where $F$
is a non-archimedean local field and $\dbG$ is a connected non-semisimple algebraic group defined over $F$;
\item[(b)] $G$ is a parabolic subgroup of $SL_n(R)$ for some infinite profinite ring $R$.
\end{itemize}
Then $G$ does not have a compactly generated, topologically simple envelope.
\end{cor}

\begin{proof} (a) By \cite[I.2.5.2(i)]{Margulis:book}, if $\dbA$
is an arbitrary algebraic group defined over $F$, then $\mathbb A(F)$
is a pure analytic manifold over $F$ of dimension $\dim\mathbb A$.
If the center of the algebraic group $\dbG$ has positive dimension, then $\Zen(G)\neq\{1\}$,
and there is nothing to prove. If $\dbG$ has finite center, let
$\mathbb U$ be the unipotent radical of $\dbG$. Since $\dbG$ is non-semisimple,
$\dim\mathbb U>0$, and thus $\mathbb U(F)\cap G$ is a non-trivial
nilpotent normal subgroup of $G$. Since $G$ is linear, the Fitting
subgroup of $G$ is nilpotent (and non-trivial), and thus we are done
by Corollary~\ref{cor:prestupid}.

The proof of (b) is analogous.
\end{proof}

\subsection{New topologically simple t.d.l.c. groups.}
\label{ss:newsim}

The majority of known examples of topologically simple t.d.l.c. groups
have a natural action on buildings. These include isotropic simple algebraic groups
over non-archimedean local fields, Kac-Moody groups over finite fields \cite{remy:simp} and
certain groups acting on products of trees \cite{bumo:ihes1}.
In \cite{claas:comm}, R\"over has shown that the commensurator of the (first)
Grigorchuk group is simple and can be described explicitly using R. Thompson's group.
We will use this result to show that the pro-$2$ completion of the Grigorchuk group
has a compactly generated topologically simple envelope. We believe that
this construction yields a new example of a topologically simple t.d.l.c. group.
\vskip .1cm

We start with a simple lemma relating the commensurator of a discrete group
to the commensurator of its profinite completion.

\begin{lem}
\label{prof_discrete}
Let $\Gamma$ be a residually finite discrete group, and let $\widehat\Gamma$
be the profinite completion of $\Gamma$. Then there is
a natural injective homomorphism $\omega_{\Gamma}:\Comm(\Gamma)\to \Comm(\Gamhat)$.
\end{lem}
\begin{proof}
Let $\Gamma_1,\Gamma_2$ be finite index subgroups of $\Gamma$
and let $\phi:\Gamma_1\to \Gamma_2$ be an isomorphism. Then
$\phi$ canonically extends to an isomorphism $\widehat\phi:\Gamhat_1\to\Gamhat_2$,
so $\widehat\phi$ is a virtual automorphism of $\Gamhat$.

If $\phi$ and $\psi$ are equivalent virtual automorphisms of $\Gamma$,
then clearly $\widehat\phi$ and $\widehat\psi$ are equivalent as well,
so there is a natural homomorphism
$\omega_{\Gamma}:\Comm(\Gamma)\to \Comm(\Gamhat)$. Finally, $\omega_{\Gamma}$
is injective because every open subgroup of $\Gamhat$ is of the
form $\widehat\Lambda$ for some finite index subgroup $\Lambda$ of $\Gamma$.
\end{proof}

Once again, let $\Gamma$ be a discrete residually finite group.
Henceforth we identify $\Comm(\Gamma)$ with its image under the homomorphism $\omega_{\Gamma}$.
Now define
\begin{equation}
\label{eq:comhat}
\widehat{\Comm}(\Gamma) = \langle\, \Comm(\Gamma),\iota(\hGamma)\,
\rangle\subseteq\Comm(\hGamma).
\end{equation}
Note that if $\VZ(\hGamma)=\{1\}$, then $\widehat{\Comm}(\Gamma)$ is an
open subgroup of the t.d.l.c. group $\Comm(\hGamma)_S$, and thus itself
a t.d.l.c. group.
\begin{thm}
\label{thm:grig}
Let $\Gamma$ be the Grigorchuk group. Then the t.d.l.c. group $\widehat{\Comm}(\Gamma)$
is compactly generated and topologically simple.
\end{thm}

\begin{proof} For discussion of various properties of $\Gamma$
and branch groups in general, the reader is referred to \cite{slaw:branch}.
The only facts we will use in this proof are the following:

(a) $\hGamma$ has trivial virtual center;

(b) $\hGamma$ is just-infinite (see Section~\ref{s:hji} for the definition);

(c) $\Comm(\Gamma)$ is a finitely presented simple group \cite[Thm.~1.3]{claas:comm}
(actually we are only using finite generation rather than finite presentation).

By (a), we can identify $\hGamma$ with $\iota(\hGamma)$.
By Proposition~\ref{prop:strong}(c), the group $\Comm(\hGamma)_S$
has trivial virtual center, and therefore, $\widehat{\Comm}(\Gamma)$
has trivial virtual center as well. Since $\widehat{\Comm}(\Gamma)$
is generated by the finitely generated group $\Comm(\Gamma)$
and the compact group $\hGamma$, it is clear that $\widehat{\Comm}(\Gamma)$
is compactly generated. It remains to show that $\widehat{\Comm}(\Gamma)$
is topologically simple.

Let $N$ be a non-trivial closed normal subgroup of $\widehat{\Comm}(\Gamma)$.
As $\widehat{\Comm}(\Gamma)$ has trivial virtual center, the group
$M= N\cap\hGamma$ is non-trivial by Proposition~\ref{prop:dis} (see next section).
As $\hGamma$ is just-infinite, this implies that $M$ is open in $\hGamma$.
Thus $N$ is open in $\widehat{\Comm}(\Gamma)$.
As $\Comm(\Gamma)$ is non-discrete in $\widehat{\Comm}(\Gamma)$,
the intersection $N\cap\Comm(\Gamma)$ is non-trivial.
Since $\Comm(\Gamma)$ is a simple group,
$N$ must contain $\Comm(\Gamma)$. In particular, $\Gamma$ is a subgroup of $N$.
Since $N$ also contains the open subgroup $M$ of $\hGamma$, it follows
that $\hGamma$ is a subgroup of $N$. Thus,
$N\supseteq \la \hGamma, \Comm(\Gamma)\ra =\widehat{\Comm}(\Gamma)$.
\end{proof}
It is well known (see \cite[Prop. 10]{slaw:branch}) that the profinite completion of the Grigorchuk group
contains every countably based pro-$2$ group.
\begin{cor}
There exists a compactly generated, topologically simple totally disconnected locally compact group
that contains every countably based pro-$2$ group.
\label{cor:embedding}
\end{cor}
In Section~\ref{s:auttop} we will prove another interesting result
about the group $\widehat{\Comm}(\Gamma)$ (where $\Gamma$ is the Grigorchuk group) --
we will show that $\widehat{\Comm}(\Gamma)$ is non-rigid in the sense of \eqref{dia:rigid2}.

\section{Commensurators of hereditarily just-infinite profinite groups}
\label{s:hji}
A profinite group $G$ is called {\it just-infinite} if it is infinite, but every
non-trivial closed normal subgroup of $G$ is of finite index.
A profinite group $G$ is called
{\it hereditarily just-infinite} (h.j.i.), if every open subgroup of $G$ is just-infinite.
Using Wilson's structure theory for the lattices of subnormal subgroups in just-infinite groups~\cite{wilson:newhor},
Grigorchuk~\cite{slaw:branch} showed that every just-infinite profinite group is either
a branch group or a finite extension of the direct product of finitely many h.j.i.
profinite groups. While the structure of branch groups appears to be very complicated, known
examples of h.j.i. profinite groups are relatively well-behaved. Furthermore, these examples
include some of the most interesting pro-$p$ groups, which makes hereditarily just-infinite
profinite groups an important class to study. Alas, very few general structure theorems about
h.j.i. profinite groups are known so far.

In this section we propose a new approach to studying h.j.i. profinite groups,
which uses the theory of commensurators. We show that all h.j.i. profinite groups
can be naturally divided into four types, based on the structure of their commensurator.
We then determine or conjecture the `commensurator type' for each of the known examples
of h.j.i. groups. This analysis leads to several interesting questions and conjectures
regarding the general structure of h.j.i. profinite groups.

\subsection{Examples of hereditarily just-infinite profinite groups.}
In this subsection we describe known examples of h.j.i. profinite groups.
At the present time all such examples happen to be virtually pro-$p$ groups, i.e.,
they contain a pro-$p$ subgroup of finite index for some prime $p$, and
it is not clear whether non-virtually pro-$p$ h.j.i. profinite groups exist.
\footnote{A large class of h.j.i. profinite groups which are not virtually pro-$p$ is
constructed in a recent paper of J. Wilson~\cite{wilson:large}, which was published
after the acceptance of the present paper.}

1. \bf{h.j.i. virtually cyclic groups. }\rm The additive group of $p$-adic integers $\dbZ_p$ is
hereditarily just-infinite, and so are some of the finite extensions of $\dbZ_p$.
It is easy to see that every h.j.i. profinite group which is virtually procyclic
(or more generally, virtually solvable) must be of this form.
We will show that a h.j.i. profinite group $G$ is virtually cyclic if and only if $\VZ(G)\not=\{1\}$.

2. \bf{h.j.i. groups of Lie type. }\rm Let $F$ be a non-archimedean local field, let
$\dbG$ be an absolutely simple simply connected algebraic group defined over $F$, and let
$L=\dbG(F)$. Then the center of $L$ is finite, and if $G$ is any open compact subgroup of $L$
such that $G\cap Z(L)=\{1\}$, then $G$ is h.j.i. If ${\rm char\, }F=0$,
this is a folklore result, and if ${\rm char\, }F=p$, it is a consequence of
\cite[Main~Theorem~7.2]{pink}.
A h.j.i. profinite group $G$ of this form will be said to have \bf{Lie type. }\rm
We will say that $G$ is of \bf{isotropic Lie type }\rm (resp. \bf{anisotropic Lie type}\rm)
if the corresponding algebraic group $\dbG$ is isotropic (resp. anisotropic) over $F$.

The groups $\SL_n(\dbZ_p)$ and $\SL_n(\Fp[[t]])$ are basic examples of h.j.i. profinite groups of isotropic Lie type.
By Tits' classification of algebraic groups over local fields, any h.j.i. profinite group of anisotropic Lie type
is isomorphic to a finite index subgroup of $\SL_1(D)$, where $D$ is a finite-dimensional central
division algebra over a local field, and $\SL_1(D)$ is the group of reduced norm one elements in $D$.

3. \bf{h.j.i. groups of Nottingham type. }\rm Recall that if $F$ is a finite field, the Nottingham group
$\ca{N}(F)$ is the group of normalized power series $\{t(1+a_1 t+\ldots) \mid a_i\in F\}$ under composition
or, equivalently, the group of wild automorphisms of the local field $F((t))$.
The following subgroups of $\ca{N}(F)$ are known to be hereditarily just-infinite:
the Nottingham group $\ca{N}(F)$ itself,
as well as three infinite families of subgroups of $\ca{N}(F)$ defined in the papers of
Fesenko \cite{fesenko}, Barnea and Klopsch \cite{yibe:nott} and Ershov \cite{mikh:new},
respectively\footnote{The groups in \cite{yibe:nott} and \cite{mikh:new} are only defined in the case when
$F$ is a prime field and, in the latter paper, under the assumption that $p>2$, but it is easy to
define analogous groups over arbitrary finite fields.}.
In addition, certain higher-dimensional analogues of the Nottingham group, called
the {\it groups of Cartan type,} are believed to be hereditarily just-infinite.
These h.j.i. groups will be said to have {\bf Nottingham type}.

\subsection{Some auxiliary results} \label{ss:normtdlc}
In this subsection we collect several results
that will be needed for our classification of h.j.i. profinite groups.

\begin{prop}
\label{prop:hjibasic}
Let $G$ be a h.j.i. profinite group. Then $G$ is virtually cyclic if and only if
$\VZ(G)\not=\{1\}$.
\end{prop}

\begin{proof}
The forward direction is obvious.
Suppose that $\VZ(G)\neq \{1\}$. Then there exists an open normal subgroup $U$ of $G$
whose centralizer in $G$ is non-trivial. Since $U$ is normal in $G$, its centralizer
$\Cent_G(U)$
is a closed normal subgroup of $G$, and thus must be open in $G$ (as $G$ is just-infinite).
Furthermore, $\Zen(U)=U\cap \Cent_G(U)$ is also open in $G$, so $\Zen(U)$ must be just-infinite.
It is clear that an abelian just-infinite profinite group must be isomorphic to $\dbZ_p$
for some prime $p$, and thus $G$ is a finite extension of $\dbZ_p$.
\end{proof}

\begin{prop}
\label{prop:dis}
Let $L$ be a t.d.l.c. group.
\begin{itemize}
\item[(a)] Let $N$ be a discrete normal subgroup of $L$. Then $N\subseteq\VZ(L)$.
In particular, if $\VZ(L)=\{1\}$, then every discrete normal subgroup of $L$ must be trivial.
\item[(b)] Assume that $\VZ(L)=\{1\}$.
Let $M$ be a non-trivial closed normal subgroup of $L$,
and let $U$ be an open subgroup of $L$. Then $M\cap U\not=\{1\}$.
\end{itemize}
\end{prop}

\begin{proof}
(a) Fix an open compact subgroup $G$ of $L$.
Take any $g\in N$.
The mapping $c_g\colon G\to N$ defined by $c_g(x)=[x,g]$
is continuous and therefore has finite image. Hence
$\Cent_G(g)=c_g^{-1}(\{1\})$ is open in $G$, and thus $g\in\VZ(L)$.

\noindent
(b) If $M\cap U=\{1\}$, then $M$ would be discrete, which is impossible by (a).
\end{proof}

\begin{prop}
\label{prop:nt}
Let $G$ be a just-infinite profinite group with $\VZ(G)=\{1\}$, and
let $N$ be a non-trivial closed normal subgroup of $\Comm(G)_S$.
Then $N$ is open in $\Comm(G)_S$.
\end{prop}

\begin{proof} Let $M= G\cap N$. Proposition~\ref{prop:dis}(b) implies that $M$ is non-trivial.
As $M$ is normal in $G$ and $G$ is just-infinite, $M$ must be open in $G$ and thus open in $\Comm(G)_S$.
Hence, $N$ is also open in $\Comm(G)_S$.
\end{proof}

\subsection{Classification of h.j.i. profinite groups by the structure of their commensurators.}

The following theorem shows that the class of h.j.i. profinite groups
is divided naturally in four subclasses (where one of the subclasses
consists of virtually cyclic groups). Recall from \S\ref{s:tdlc}.1 that
for a t.d.l.c. group $L$ we set $\Ht^2(L)=\Ht(\Ht(L))$.

\begin{thm}
\label{thm:hji}
Let $G$ be a h.j.i. non-virtually cyclic profinite group.
Then precisely one of the following conditions holds:
\begin{itemize}
\item[(i)] $\Ht(\Comm(G)_S)=\{1\}$, so $\Comm(G)_S$ is pro-discrete;
\item[(ii)] $\Ht(\Comm(G)_S)$ is an open topologically simple subgroup of $\Comm(G)_S$;
\item[(iii)] $\Ht(\Comm(G)_S)\not=\{1\}$, but $\Ht^2(\Comm(G)_S)=\{1\}$.
\end{itemize}
\end{thm}

\begin{proof} Suppose that $C\deq\Ht(\Comm(G)_S))$ is non-trivial.
Then by Proposition~\ref{prop:nt}, $C$ is also open in $\Comm(G)_S$.
Moreover, $\Ht(C)$ is a closed, characteristic subgroup of $C$, and thus also
normal in $\Comm(G)_S$. Thus, either $\Ht(C)=\{1\}$ and (iii) holds,
or $\Ht(C)$ is open in $\Comm(G)_S$. In the latter case $C$ must be equal to
$\Ht(C)$. It remains to show that the equality
$C=\Ht(C)$ implies that $C$ is topologically simple.

Since $\VZ(G)=\{1\}$ by Proposition~\ref{prop:hjibasic},
we identify $G$ with $\iota(G)\subseteq\Comm(G)_S$.
By Proposition~\ref{prop:nt}, the group $O\deq G\cap C$ is an open subgroup
of $G$, and thus in particular a h.j.i. profinite group.
Let $N$ be a non-trivial closed normal subgroup of $C$. As $\VZ(C)=\{1\}$,
the group $M\deq N\cap O$ is non-trivial by Proposition~\ref{prop:dis}(b).
Since $M$ is a closed normal subgroup of $O$
and $O$ is just-infinite, $M$ is open in $O$ and thus in $C$.
Hence $N$ is also open in $C$. However, as $C$ coincides with its open normal core,
the only open normal subgroup of $C$ is $C$ itself.
\end{proof}

In view of Theorem~\ref{thm:hji}, we introduce the following definition:
\begin{defi}
Let $G$ be a profinite group $G$ with $\VZ(G)=\{1\}$. We say that $G$ is of
\begin{itemize}
\item[] {\it pro-discrete type } if $\Comm(G)_S$ is pro-discrete,
\item[] {\it simple type } if $\Ht(\Comm(G)_S)$ is topologically simple
(so, in particular, $\Ht(\Comm(G)_S)=\Ht^2(\Comm(G)_S)$),
\item[] {\it mysterious type } if $\Ht(\Comm(G)_S)\not=\{1\}$
and $\Ht^2(\Comm(G)_S)=\{1\}$.
\end{itemize}
\end{defi}

According to Theorem \ref{thm:hji}, every non-virtually cyclic h.j.i. profinite group
has one of the three commensurator types defined above. We shall now state
or conjecture the commensurator type for each of the known examples of h.j.i.
profinite groups.

1. Let $G$ be a h.j.i. group of isotropic Lie type. Then $G$ is of simple type
by Theorem~\ref{Chevalley}.

2. Let $G$ be a h.j.i. group of anisotropic Lie type. Then $\Comm(G)$ is a finite extension of $G$
by Theorem~\ref{Chevalley}, and thus $G$ is of pro-discrete type.

3. In \cite{ersh:nott}, it is shown that for $p>3$,
the commensurator of the Nottingham group
$\nott(\F_p)$ coincides with $\Aut(\Fpt)$. In particular,
$\nott(\F_p)$ is a finite index subgroup of its commensurator,
and therefore $\nott(\F_p)$ is of pro-discrete type. We expect
that all h.j.i. groups of Nottingham type are of pro-discrete type.
\vskip .12cm
An example of a profinite group of mysterious type is $G=\dbZ_p^*\ltimes \dbZ_p$.
Indeed, it is easy to see that $\Comm(G)_S$ is isomorphic to
$\dbQ_p^*\ltimes \dbQ_p$ (with natural topology), so $\Ht(\Comm(G)_S)\cong \dbQ_p$
and $\Ht^2(\Comm(G)_S)=\{1\}$. However, we are not aware of any examples of
h.j.i. groups of mysterious type.

\begin{quest}
\label{qu:non}
Does there exist a h.j.i. profinite group of mysterious type?
\end{quest}

Another important question is whether there are any currently unknown h.j.i. profinite
groups of simple type:

\begin{quest}
\label{qu:clsim}
Let $G$ be a h.j.i. profinite group of simple type.
Is it true that $G$ is of isotropic Lie type?
\end{quest}

An affirmative answer to this question would provide
a purely group-theoretic characterization of h.j.i. groups of isotropic Lie type.
It might be easier to answer Question~\ref{qu:clsim} in the affirmative
if we assume that $G$ is a pro-$p$ group.
\vskip .12cm
Finally, we prove a peculiar result showing that one can prove
that a h.j.i. profinite group $G$ is of simple type without computing $\Comm(G)$:

\begin{prop}
\label{prop:hjisim}
Let $G$ be a non-virtually cyclic h.j.i. profinite group, and suppose that $G$
has a topologically simple envelope $L$.
Then $G$ is of simple type.
\end{prop}

\begin{proof} By hypothesis, $\VZ(G)=\{1\}$.
By Proposition \ref{prop:unienv}, we may assume that $L$ is an open subgroup of $\Comm(G)_S$.
Every open normal subgroup of $\Comm(G)_S$ has non-trivial intersection with $L$,
and thus must contain $L$. Hence $C\deq\Ht(\Comm(G)_S)$ is non-trivial.
Moreover, $L$ is an open subgroup of $C$, so we can repeat the above argument with
$\Comm(G)_S$ replaced by $C$. It follows that $\Ht(C)\not=\{1\}$, and therefore
$G$ is of simple type.
\end{proof}

Note that Proposition~\ref{prop:hjisim} implies that h.j.i. groups
of isotropic Lie type are also of simple type independently of Theorem~\ref{Chevalley}.
Indeed, if $\dbG$ is an isotropic absolutely simple simply connected algebraic
group over a local field $F$, and $G$ is an open compact subgroup of $\dbG(F)$,
with $\VZ(G)=\{1\}$, then $L=\dbG(F)/\Zen(\dbG(F))$ is a topologically simple envelope of $G$,
and thus $\Ht(\Comm(G)_S)$ is simple by Proposition~\ref{prop:hjisim}. However, Theorem~\ref{Chevalley}
tells us more, namely that $\Ht(\Comm(G)_S)$ is equal to $L$, and one might
ask if this is an indication of a general phenomenon:

\begin{quest}
\label{qu:ex}
Let $G$ be a h.j.i. profinite group of simple type. Is it true
that any topologically simple envelope $L$ of $G$ coincides with $\Ht(\Comm(G)_S)$?
\end{quest}

It might be easier to answer Question~\ref{qu:ex} under additional assumptions such as
`$G$ is pro-$p$', `$L$ is compactly generated' or `$\Ht(\Comm(G)_S)$ is compactly generated'.

\section{Commensurators of absolute Galois groups}
\label{s:galois}
In this section we assume that $F$ is a field and denote by
$G_F=\Gal(F^{\sep}/F)$ its absolute Galois group.

\subsection{Hyperrigid fields and the Neukirch-Uchida property}
\label{ss:neucprop}
It is a common convention to say that a field $F$ has property $\eu{X}$, if its absolute Galois group $G_F$ has property $\eu{X}$.
Thus, a field $F$ will be called
{\it hyperrigid} if the canonical map
$\iota_{G_F}\colon G_F\to\Comm(G_F)_S$
is an isomorphism (see \S~\ref{ss:unienv}).

Let $E_1$ and $E_2$ be fields, and let $E^{\sep}_1$ and $E^{\sep}_2$ be separable
closures of $E_1$ and $E_2$, respectively. We define
\begin{equation}
\label{eq:defiso1}
\Iso(E^{\sep}_1/E_1,E^{\sep}_2/E_2)=\{\,\alpha\colon E_1^{\sep}\overset{\sim}{\longrightarrow}
E^{\sep}_2\mid \alpha(E_1)=E_2\,\}
\end{equation}
to be the set of all isomorphism from $E_1^{\sep}$ to $E_2^{\sep}$
which map $E_1$ to $E_2$. If $E_1$ and $E_2$ are extensions of a field $F$,
we put
\begin{equation}
\label{eq:defiso2}
\Iso_F(E_1^{\sep}/E_1,E_2^{\sep}/E_2)=\{\,\alpha\colon E_1^{\sep}\overset{\sim}{\longrightarrow}
E^{\sep}_2\mid \alpha(E_1)=E_2\ \text{and}\ \alpha\vert_F=\iid_F\,\}.
\end{equation}
Note that if $F$ is a prime field, one has a canonical bijection between $\Iso(E^{\sep}_1/E_1,E^{\sep}_2/E_2)$ and
$\Iso_F(E^{\sep}_1/E_1,E^{\sep}_2/E_2)$.

Every isomorphism of fields $\alpha\in\Iso(E^{\sep}_1/E_1,E^{\sep}_2/E_2)$ induces
the corresponding isomorphism of Galois groups
\begin{equation}
\label{eq:isoind}
\alpha_\ast\colon\Gal(E_1^{\sep}/E_1)\longrightarrow\Gal(E^{\sep}_2/E_2)
\end{equation}
given by $\alpha_\ast(g)(x)=\alpha(g(\alpha^{-1}(x)))$
for $g\in\Gal(E^{\sep}_1/E_1)$ and $x\in E^{\sep}_2$.

\begin{defi}
Let $F$ be a field. We say that $F$ has the {\it Neukirch-Uchida property}
if the following holds:
Let $E_1/F$ and $E_2/F$ be finite separable extensions of $F$. Then for every
isomorphism
$\sigma\colon\Gal(E_1^{\sep}/E_1)\to\Gal(E^{\sep}_2/E_2)$ of profinite groups
there exists a unique
element $\alpha\in\Iso_F(E_1^{\sep}/E_1,E^{\sep}_2/E_2)$ such that
$\sigma=\alpha_\ast$.
\end{defi}

\begin{prop}
\label{prop:neuc}
Let $F$ be a field. Then $F$ has the Neukirch-Uchida property if and only if
$F$ is hyperrigid.
\end{prop}

\begin{proof}
We fix a separable closure $F^{\sep}$ of $F$.
Assume that $F$ has the Neukirch-Uchida property.
Let $\phi\colon U\to V$ be a virtual automorphism of $G_F=\Gal(F^{\sep}/F)$.
Let $E_1=(F^{\sep})^U$ and $E_2=(F^{\sep})^V$. Then $\phi$ is a continuous
isomorphism from $\Gal(F^{\sep}/E_1)=U$ to $\Gal(F^{\sep}/E_2)=V$.
Hence by definition, there exists a unique element
\begin{equation}
\label{eq:propneuc}
g\in\Iso_F(F^{\sep}/E_1,F^{\sep}/E_2)=
\{\,y\in G_F\mid yUy^{-1}=V\,\}
\end{equation}
such that $i_g\vert_U=\phi$. Hence $G_F$ is hyperrigid.

Suppose $G_F$ is hyperrigid.
Let $E_1$ and $E_2$ be finite separable extensions of $F$, and let $E_1^{\sep}$ and
$E_2^{\sep}$ be  separable closures of $E_1$ and $E_2$, respectively.
We also fix a separable closure $F^{\sep}$ of $F$ and two isomorphisms
$\rho_1\colon F^{\sep}\to E_1^{\sep}$ and
$\rho_2\colon F^{\sep}\to E_2^{\sep}$
which fix $F$ pointwise. Let $E_1'=\rho_1^{-1}(E_1)$ and $E_2'=\rho_2^{-1}(E_2)$.

Let $\alpha\colon\Gal(E_1^{\sep}/E_1)\to\Gal(E_2^{\sep}/E_2)$ be an isomorphism
of profinite groups.
As $\Gal(F^{\sep}/F)$ is hyperrigid, there exists a unique element $g\in\Gal(F^{\sep}/F)$
such that the diagram
\begin{equation}
\label{dia:neuc}
\xymatrix{
\Gal(F^{\sep}/E_1')\ar[r]^{i_g}\ar[d]_{(\rho_1)_\ast}&
\Gal(F^{\sep}/E_2')\ar[d]^{(\rho_2)_\ast}\\
\Gal(E_1^{\sep}/E_1)\ar[r]^\alpha&\Gal(E_2^{\sep}/E_2)
}
\end{equation}
commutes. Then $\sigma=\rho_2\circ g\circ\rho_1^{-1}\in
\Iso_F(E_1^{\sep}/E_1, E_2^{\sep}/E_2)$ and $\sigma_\ast=\alpha$.
The uniqueness of $g$ implies the uniqueness of $\sigma$.
Thus $F$ has the Neukirch-Uchida property.
\end{proof}

In \cite{neu:138}, \cite{neu:139} and \cite{uch:211},
it was proved that $\Q$ has the Neukirch-Uchida property.
In view of Proposition~\ref{prop:neuc}, this result can be reformulated as follows:

\begin{thm}[(Neukirch \& Uchida)]
\label{thm:neuc2}
Let $F$ be a number field, i.e., $F$ is a finite extension of $\Q$. Then
the canonical mapping $j_F\colon G_{\Q}\to\Comm(G_F)_S$ is an isomorphism.
\end{thm}

\subsection{Anabelian fields}
\label{ss:anab} Following \cite[Chap.~XII]{NSW:coh} we call a field $F$ {\it anabelian},
if $\VZ(G_F)=1$.
Thus, if $F$ is an anabelian field, $\Comm(G_F)_S$ is a t.d.l.c. group.
Obviously, finite fields as well as the real field are not anabelian.
The simplest examples of anabelian fields are the $p$-adic fields $\dbQ_p$.
This is a consequence of the following result:

\begin{prop}
\label{prop:scd}
Let $G$ be a profinite group satisfying $\cd(G)=\scd(G)=2$,
where $\cd(G)$ (resp. $\scd(G)$) denotes the cohomological dimension of $G$
(resp. the strict cohomological dimension of $G$).
Then $\VZ(G)=\{1\}$.
\end{prop}

\begin{proof}
Let $g\in\VZ(G)$, with $g\not=1$. Then $U=\Cent_G(g)$ is open and thus of finite
index in $G$. In particular, $\scd(U)=\cd(U)=2$ and $g\in \Zen(U)$. Hence
we may replace $G$ by $U$ and thus assume that $\Zen(G)\neq \{1\}$.

As $\Zen(G)$ is closed, it is a profinite group. Note that
$\Zen(G)$ is torsion-free since $\cd(\Zen(G))\leq \cd(G)<\infty$.
Since $\Zen(G)$ is also abelian, we can find a subgroup $C\subseteq\Zen(G)$
such that $C\cong \dbZ_p$ for some prime $p$.

Let $P$ be a Sylow pro-$p$ subgroup of $G$. Then
$\scd_p(P)=\cd_p(P)=2$ (see \cite[\S I.3.3]{ser:gal}) and
$C\subseteq \Zen(P)$. From \cite{weza:coh} one concludes that
$\vcd_p(P/C)=1$. In particular, $P/C$ is not torsion, so
$P$ contains a closed subgroup isomorphic to $\Z_p\times\Z_p$.
However, $\scd_p(\Z_p\times\Z_p)=3>\scd_p(P)$, a contradiction.
\end{proof}

It is well-known that $\scd(G_{\Q_p})=\cd(G_{\Q_p})=2$ (see \cite[\S II.5.3]{ser:gal}).
Thus, Proposition~\ref{prop:scd} implies that $\VZ(G_{\Q_p})=1$, so
$\Comm(G_{\Q_p})_S$ is a t.d.l.c. group. Mochizuki's theorem (see \cite{mochi:qp})
suggests that its structure should be related to the ramification filtrations on the
group $G_{\Q_p}$ and its open subgroups, but the following questions remain open.

\begin{quest}
\label{qu:padic}
\begin{itemize}
\item[(I)] What is the structure of the t.d.l.c. group $\Comm(G_{\Q_p})_S$?
\item[(II)] Is $\Comm(G_{\Q_p})$ strictly larger than $\Aut(G_{\Q_p})$?
\item[(III)] Is $\Comm(G_{\Q_p})_S$ a t.d.l.c. group with a non-trivial scale function?
\end{itemize}
\end{quest}

\begin{rem} In \cite[Chap.~VII, \S 5]{NSW:coh}, it is shown that $G_{\Q_p}$ has non-trivial
outer automorphisms, and thus $\Q_p$ does not possess the Neukirch-Uchida
property. So part (II) of Question~\ref{qu:padic} asks whether one can construct elements in $\Comm(G_{\Q_p})_S$
outside the normalizer of $\image(\iota_{G_{\Q_p}})$.

If $U$ and $V$ are isomorphic open subgroups of $G_{\Q_p}$, then
$U$ and $V$ must be of the same index in $G_{\Q_p}$. This follows from the fact
that the (additive) Euler characteristic of $G_{\Q_p}$ at the prime $p$ is $-1$.
Hence $\Comm(G_{\Q_p})_S$ is unimodular by Proposition~\ref{prop:strong}(d).
However, it is not clear to us whether $\Comm(G_{\Q_p})_S$ is
uniscalar or not.
\end{rem}

\subsection{Totally disconnected locally compact groups arising from finitely generated
field extensions}
\label{ss:high}
For a field $F$ and a non-negative integer $n$ we define
\begin{equation}
\label{eq:defGFn}
G_F(n)=\Aut_F(\rat{F}{n}^{\sep}).
\end{equation}
Consider the field extension $\rat{F}{n}^{\sep}/F$. Let $\bE$ denote the set
of subfields $E$ of $\rat{F}{n}^{\sep}$ with the following properties
\begin{itemize}
\item[(i)] $F\leq E$ and $E$ is finitely generated over $F$,
\item[(ii)] $\rat{F}{n}^{\sep}/E$ is a separable extension.
\end{itemize}
Let $E$, $E^\prime\in\bE$. Then $E\vee E^\prime$
-- the subfield of $\rat{F}{n}^{\sep}$
generated by $E$ and $E^\prime$ -- is also contained in $\bE$. Moreover,
if $g\in G_F(n)$, then $g(E)$ is also contained in $\bE$.
Hence the set of subgroups
\begin{equation}
\label{eq:deftop}
\ca{F}(\bE)=\{\,G_E=\Aut_E(\rat{F}{n}^{\sep})\mid E\in\bE\,\}
\end{equation}
satisfies the hypothesis of Proposition \ref{prop:topbou}, and thus
defines a unique topology $\ca{T}$ making $G_F(n)$  a topological group
for which $\ca{F}(\bE)$ is a base of neighborhoods of the identity. In particular,
on every subgroup $G_E$, the induced topology coincides with the
Krull topology. Hence $G_F(n)$ is a t.d.l.c. group.

\subsection{Compact subgroups of $G_F(n)$}
\label{ss:compsep}

For the analysis of compact subgroups of $G_F(n)$
we shall use the following well-known result due to Artin
(see \cite[Chap.~6]{Lang:Algebra}):

\begin{prop}
\label{fact:comp}
Let $E$ be a field, and let $G$ be a finite subgroup of $\Aut(E)$.
Let $E_0=E^G$ be the fixed field of $G$. Then $E/E_0$ is a finite Galois extension
with Galois group $G$.
\end{prop}

\begin{prop}
\label{prop:comp}
Let $C$ be a compact subgroup of $G_F(n)$, and define
\begin{equation}
\label{eq:corcom}
\bo{F}(C)=(\rat{F}{n}^{\sep})^C=\{\,y\in\rat{F}{n}^{\sep}\mid c(y)=y\ \text{for all
$c\in C$}\,\}
\end{equation}
Then $\bo{F}(C)$ is a subfield containing $F$, and the extension $\rat{F}{n}^{\sep}/\bo{F}(C)$ is separable.
\end{prop}

\begin{proof} Let $E=\rat{F}{n}$, and let
$O=\Aut_{E}(\rat{F}{n}^{\sep})$. Then $O$ is an open compact
subgroup of $G_F(n)$. As $C\subseteq G_F(n)$ is compact, $C\cap O$ has finite index in $C$.
In particular, $O^\prime=\bigcap_{c\in C} O^c$ is of finite index in $O$.
By construction, $C^\prime=C\cap O^\prime$ is an open normal subgroup of $C$
and a closed subgroup of $O$.
Let $E^\prime$ denote the fixed field of $C^\prime$.
Then one has a canonical injection
$\iota\colon C/C^\prime\to\Aut_{F}(E^{\prime})$.
Let $E_0=(E^{\prime})^C$. By construction, $\bo{F}(C)=E_0$.
By Proposition~\ref{fact:comp}, $E^{\prime}/E_0$ is a finite separable extension.
Hence $\rat{F}{n}^{\sep}/E_0$ is separable.
\end{proof}

As a consequence of Proposition~\ref{prop:comp} we obtain the following variation
of the Fundamental theorem in Galois theory.

\begin{thm}
\label{thm:main}
Let $\Com$ denote the set of compact subgroups of $G_F(n)$, and let
$\interm$ denote the set of subfields $E$ of $\rat{F}{n}^{\sep}$ containing $F$
such that the field extension $\rat{F}{n}^{\sep}/E$ is separable. Then
the maps
\begin{equation}
\label{eq:main}
\begin{aligned}
\bo{A}(\argu)=&
\Aut_{\argu}(\rat{F}{n}^{\sep})&\colon&\interm&\longrightarrow&\quad\Com\\
\bo{F}(\argu)=&
(\rat{F}{n}^{\sep})^{\argu}&\colon&\Com&\longrightarrow&\quad\interm
\end{aligned}
\end{equation}
are mutually inverse, i.e., $\bo{A}\circ\bo{F}=\iid_{\Com}$ and
$\bo{F}\circ\boA=\iid_{\interm}$. Moreover, if $E\in\interm$, then
$\boA(E)$ is compact and open if and only if $E$ is finitely generated over $F$.
\end{thm}

\begin{proof} The Fundamental Theorem in Galois theory implies that the mappings $\boA$ and $\boF$ are mutually inverse.

If $E\in\bE$, then $\boA(E)$ is open by definition. Assume that $\boA(E)$ is open
in $G_F(n)$. Let $E^\prime\in\bE$. Then $\boA(E)\cap\boA(E^\prime)=
\boA(E\vee E^\prime)$ is open and thus of finite index in $\boA(E^\prime)$.
Hence $E\vee E^\prime$ is finitely generated over $F$.
Moreover, $E\vee E^\prime/E$ is a finite separable extension, and thus
$E$ is finitely generated over $F$ (see \cite{roman:field}).
Hence $E\in\bE$.
\end{proof}

\subsection{Finitely generated extensions of $\Q$}
\label{ss:pop}
In \cite{pop:157} and \cite{pop:158},
Pop extended the Neukirch-Uchida theorem to
fields which are finitely generated over $\Q$. His result can be reformulated
using the same ideas as in \S\ref{ss:neucprop}:

\begin{thm}[(Pop)]
\label{thm:popup}
The t.d.l.c. group $G_{\Q}(n)$ is hyperrigid for every $n\geq 1$.
In particular, if $E$ is a field which is finitely generated over $\Q$
of transcendence degree $n$, then $E$ is anabelian, and
the canonical map $j_E\colon G_{\Q}(n)\to\Comm(G_E)_S$ is an isomorphism.
\end{thm}

\section{The Aut-topology}
\label{ss:topaut}

\subsection{Automorphisms of profinite groups}
\label{s:auto}
Let $G$ be a profinite group, and let $\Aut(G)$ denote the group of continuous automorphisms of $G$.
The {\it standard } topology on  $\Aut(G)$ is given by the base
$\{\,A(O) \mid \text{$O$  open in  $G$}\,\}$
of neighborhoods of $1\in\Aut(G)$, where
\begin{equation}
\label{eq;aut1}
A(O)=\{\,\gamma\in \Aut(G) \mid
\text{$\gamma(g)\equiv g \mod O$ for every $ g\in G$}\,\}.
\end{equation}
With this topology, $\Aut(G)$ is always a Hausdorff topological group but not
necessarily profinite.

To ensure that $\Aut(G)$ is profinite, we need to require that
$G$ is {\it characteristically based}, that is,
$G$ has a base $\ca{C}$ of neighborhoods of the identity consisting of characteristic subgroups.
In this case, there exists a canonical isomorphism
\begin{equation}
\label{eq:aut3}
\Aut(G)\simeq {\textstyle \varprojlim_{C\in\ca{C}} \Aut(G/C)},
\end{equation}
of topological groups \cite[Prop.~5.3]{ddms:padic}.
In particular, $\Aut(G)$ is an inverse limit of finite groups and thus profinite.

A sufficient condition for $G$ to be characteristically based is that $G$ is
(topologically) {\it finitely generated},
that is, contains an (abstract) finitely generated dense subgroup.
Indeed, if $G$ is finitely generated, then for every
positive integer $n$, there are only finitely many open subgroups of index $n$ in $G$.
Their intersection $C_n$ is an open characteristic subgroup of $G$,
which is contained in every open subgroup of $G$ of index $n$. Thus $G$ has a
base consisting of characteristic subgroups; furthermore, this base is countable.

\subsection{Hereditarily countably characteristically based profinite groups}
\label{ss:hccb}
A profinite group $G$ will be called {\it countably characteristically based}
if it has a countable base $\{G_i\}_{i=1}^{\infty}$ where each $G_i$ is
a characteristic subgroup. We will say that $G$ is
{\it hereditarily countably characteristically based} (h.c.c.b.),
if every open subgroup of $G$ is countably characteristically based.
Profinite groups with this property can be characterized as follows:

\begin{prop}
\label{prop:hccb1}
Let $G$ be a profinite group. The following are equivalent:
\begin{itemize}
\item[(i)] $G$ is hereditarily countably characteristically based;
\item[(ii)] $G$ is countably based, and for every pair of open subgroups $U$ and $V$
of $G$, with $V\subseteq U$,
there exists an open subgroup $W\subseteq V$ which is characteristic in $U$.
\end{itemize}
\end{prop}

As we showed in the previous subsection, every finitely generated profinite
group is countably characteristically based. Since a finite index subgroup of a finitely
generated group is finitely generated, we have the following implication:
\begin{equation}
\label{eq:aut4}
\text{finite generation}\Longrightarrow{h.c.c.b.}
\end{equation}
\vskip .2cm

Let $U$ be an open subgroup of a profinite group $G$.
We set
\begin{equation}
\label{eq:aut5}
\Aut(G)_U=\{\,\alpha\in\Aut(G)\mid \alpha(U)=U\,\},
\end{equation}
and $$r_{G,U}\colon \Aut(G)_U\to \Aut(U)$$ will denote the restriction
map. Note that $U$ is characteristic in $G$ if and only if
$\Aut(G)_U=\Aut(G)$.

\begin{prop}
\label{prop:hccb2}
Let $G$ be a characteristically based profinite group, and let
$U$ be an open subgroup of $G$. Then $\Aut(G)_U$ is an open
subgroup of $\Aut(G)$, and the restriction map
$r_{G,U}\colon\Aut(G)_U\to \Aut(U)$ is continuous.
\end{prop}

\begin{proof}
Let $W$ be an open characteristic subgroup of $G$
such that $W\subseteq U$. Then $A(W)=\kernel(\Aut(G)\to\Aut(G/W))$ is open in
$\Aut(G)$ and also contained in $\Aut(G)_U$.
Therefore, $\Aut(G)_U$ is open in $\Aut(G)$ as well.
Continuity of the map $r_{G,U}$ is proved in a similar way.
\end{proof}

\subsection{The Aut-topology}
Let $G$ be a h.c.c.b. profinite group.
In this case there is a natural topology on $\Comm(G)$
- the {\it Aut-topology} -
which is `compatible' with the standard topologies on the groups
$\Aut(U)$, with $U$ open in $G$. More precisely, $\Aut$-topology
can be characterized as the strongest topology $\ca{T}$ such that
\begin{itemize}
\item[(i)] all maps $\rho_U\colon\Aut(U)\to (\Comm(G),\ca{T})$ are continuous,
\item[(ii)] $\ca{T}$ has a base of neighborhoods of $1_{\Comm(G)}$ consisting of open subgroups.
\end{itemize}
If $\ca{T}$ is any topology satisfying (i) and (ii), then
all $\ca{T}$-open subgroups must belong to the set $\ca{B}_G$
where
\begin{multline}
\label{eq:Bg}
\ca{B}_G=\{\,H\subseteq\Comm(G)\mid \text{ $H$ is a subgroup  and}\\
\text{$\rho_O^{-1}(H)$ is open in $\Aut(O)$ for every open subgroup $O$ of $G$.}\}
\end{multline}

Thus, the strongest topology satisfying (i) and (ii) can be defined as follows:

\begin{defi}
The {\it Aut-topology} on $\Comm(G)$ -- denoted by $\ca{T}_A$ --
is the unique topology such that $\ca{B}_G$ is a base of neighborhoods of
$1_{\Comm(G)}$.
The topological group $(\Comm(G),\ca{T}_A)$ will be denoted by $\Comm(G)_A$.
\end{defi}

In order to prove that the $\Aut$-topology is well defined and turns $\Comm(G)$
into a topological group, we will show that the set $\ca{B}_G$ satisfies
the hypotheses of Proposition~\ref{prop:topbou}. Furthermore, we will show that
the topological group $\Comm(G)_A$ depends only on the commensurability class of $G$.

\begin{prop}
\label{prop:BG}
Let $G$ be a h.c.c.b. profinite group. Then $\Comm(G)_A$
is a topological group. Furthermore, if $U$ is an open subgroup of $G$,
the natural map $j_{U,G}\colon\Comm(U)_A\to\Comm(G)_A$ is a homeomorphism.
\end{prop}

The proof of Proposition~\ref{prop:BG} will be based on the following simple lemma.

\begin{lem}
\label{lemma:BG}
Let $G$ be a h.c.c.b. profinite group. Let $U,V$ be open subgroup of $G$,
with $U\subseteq V$. If $H$ is a subgroup
of $\Comm(G)$ such that $\rho_V^{-1}(H)$ is open in $\Aut(V)$, then
$\rho_U^{-1}(H)$ is open in $\Aut(U)$.
\end{lem}

\begin{proof}
The restriction of the map $\rho_U:\Aut(U)\to\Comm(G)$ to $\Aut(U)_V$
coincides with the composition $\rho_{V}\circ r_{U,V}\colon\Aut(U)_V\to \Comm(G)$.
Since $\rho_{V}^{-1}(H)$ is open in $\Aut(V)$ and $r_{U,V}$ is continuous, we
conclude that $\rho_{U}^{-1}(H)\cap \Aut(U)_V=(\rho_{V}\circ r_{U,V})^{-1}(H)$
is open in $\Aut(U)_V$. Since $\Aut(U)_V$ is open in $\Aut(U)$, it follows
that $\rho_{U}^{-1}(H)$ is open in $\Aut(U)$ as well.
\end{proof}

\begin{proof}[Proof of Proposition~\ref{prop:BG}]
The inclusion $\ca{B}_G\subseteq \ca{B}_U$ is obvious, and
Lemma~\ref{lemma:BG} implies that $\ca{B}_U\subseteq \ca{B}_G$.
Thus, the bases $\ca{B}_U$ and $\ca{B}_G$ coincide, and
the mapping $j_{U,G}:\Comm(U)_A \to \Comm(G)_A$ is a homeomorphism
of topological spaces. It remains to show that
$\Comm(G)_A$ is a topological group.

It is clear that the set $\ca{B}_G$ is closed under intersections.
By Proposition~\ref{prop:topbou}, in order prove that $\Comm(G)_A$
is a topological group, it suffices to show that $\ca{B}_G$ is invariant under conjugation.
Let $H\in\ca{B}_G$, let $\phi$ be a virtual automorphism of $G$, and let $U$ be an open
subgroup of $G$ on which both $\phi$ and $\phi^{-1}$ are defined.
We will show that $\rho_V^{-1}([\phi]^{-1} H[\phi])$ is open for every
subgroup $V$ which is open in $U$.
Since $\ca{B}_G=\ca{B}_U$, this will imply that $[\phi]^{-1}H[\phi]$ lies in $\ca{B}_G$.

For every $V$ open in $U$ we
have a commutative diagram
\begin{equation}
\label{dia:commaut1}
\xymatrix{
\Aut(V)\ar[r]^{\rho_V}\ar[d]_{i_\phi}&\Comm(G)\ar[d]^{i_{[\phi]}}\\
\Aut(\phi(V))\ar[r]^{\rho_{\phi(V)}}&\Comm(G)
}
\end{equation}
where $i_\phi\colon \Aut(V)\to \Aut(\phi(V))$ is left conjugation by
$\phi$, that is, $i_\phi (\psi)=\phi \psi \phi^{-1}$, and
similarly, $i_{[\phi]}$ is left conjugation by $[\phi]$.
Therefore, we have
\begin{equation}
\label{eq:commaut2}
\begin{aligned}
\rho_V^{-1}([\phi]^{-1} H[\phi])&=
\rho_V^{-1}\circ i_{[\phi]}^{-1}(H)=
(i_{[\phi]}\circ\rho_V)^{-1}(H)\\
&=(\rho_{\phi(V)}\circ i_\phi)^{-1}(H)=
(i_\phi)^{-1} \rho_{\phi(V)}^{-1}(H)
\end{aligned}
\end{equation}
Since $H\in\ca{B}_G$ and $i_\phi^{-1}$ is continuous, we conclude
that $\rho_V^{-1}([\phi]^{-1} H[\phi])$ is open.
\end{proof}

Lemma~\ref{lemma:BG} yields a simple characterization of open subgroups of $\Comm(G)_A$:

\begin{claim}
\label{claim:critopen}
Assume that $G$ is a h.c.c.b. profinite group.
Let $\{G_i\}_{i\in\N}$ be a base of neighborhoods of the identity in $G$,
where each $G_i$ is a subgroup.
Then a subgroup $H$ of $\Comm(G)_A$ is open if and only if
$\rho_{G_i}^{-1}(H)$ is open in $\Aut(G_i)$ for all $i\in\N$.
\end{claim}

In spite of the above criterion, it may not be clear so far how to construct
open subgroups of $\Comm(G)_A$. We will now describe explicitly a large
family of open subgroups which are canonically associated to certain
bases of $G$.

\begin{defi}
Let $G$ be a profinite group. A countable sequence $\ca{F}=\{G_k\}_{k\in\N}$
of open subgroups of $G$ will be called a {\it super-characteristic base},
if $\ca{F}$ is a base of neighborhoods of $1_G$ and $G_{k+1}$ is characteristic in
$G_k$ for each $k\in\N$.
\end{defi}

It is clear that any h.c.c.b. profinite group has a super-characteristic base.

\begin{prop}
\label{prop:pureaut}
Let $\{G_i\}_{i\in\N}$ be a super-characteristic base of a
h.c.c.b. profinite group $G$.
For $i\in\N$ let $A_i=\Autbar(G_i)$, and put $A=\bigcup_{i\in\N}A_i$.
Then $A$ is an open subgroup of $\Comm(G)_A$, and
the index $[\Comm(G)\colon A]$ is countable (where by countable
we mean finite or countably infinite).
\end{prop}

\begin{proof}
Since $G_{i+1}$ is characteristic in $G_i$, we have $A_{i}\subseteq A_{i+1}$ for each $i$,
and therefore $A$ is a subgroup. By Claim~\ref{claim:critopen}, $A$ is open
since $\rho_{G_i}^{-1}(A)=\Aut(G_i)$ for $i\in\N$.
Finally, we claim that the index $[\Comm(G)\colon A]$
is countable. Indeed, every virtual automorphism $\phi$ of $G$ is
defined on $G_i$ for some $i$, and since $G$ is countably based, there are
only countably many choices for $\phi(G_i)$. If $\psi$ is another virtual
automorphism of $G$ defined on $G_i$ and such that $\phi(G_i)=\psi(G_i)$,
then $[\phi]^{-1}[\psi]\in \Autbar(G_i)$.
\end{proof}

\section{Further properties of the Aut-topology}
\label{s:auttop}

Let $G$ be a h.c.c.b. profinite group. In this section we determine when the topological
group $\Comm(G)_A$ is Hausdorff and when $\Comm(G)_A$ is locally compact. The answers to
both questions are expressed in terms of certain finiteness conditions on the
\it{automorphism system }\rm of $G$, that is, the family of groups
$\{\Aut(U)\mid U \mbox{ is open in }G\}$ along with the maps
$\{r_{U,V}:\Aut(U)_V\to\Aut(V) \mbox{ for }V\subseteq U\}$.
\vskip .12cm
Recall that if $U$ is an open subgroup of $G$, then
$\Autbar(U)$ denotes the  image of the canonical map $\rho_U:\Aut(U)\to\Comm(G)$.
Note that $\Autbar(U)$ is isomorphic to $\Aut(U)/\TAut(U)$,
and thus has the natural quotient topology. This quotient topology is
Hausdorff if and only if $\TAut(U)$ is closed in $\Aut(U)$. Thus, we
are led to consider the following finiteness condition:

\begin{defi} A profinite group $G$ is said to be $\Aut_1$
if $\TAut(U)$ is closed in $\Aut(U)$ for every open subgroup $U$ of $G$.
\end{defi}

We will show (see Theorem~\ref{thm:t1}) that a h.c.c.b. profinite group $G$ is  $\Aut_1$ if and only if $\Comm(G)_A$ is Hausdorff.

The second finiteness condition we introduce is ``virtual stabilization''
of the automorphism system of $G$. If $U,V$ are open subgroups of $G$,
with $V\subseteq U$, the restriction map $r_{U,V}:\Aut(U)_V\to\Aut(V)$
induces the embedding $\Autbar(U)_V\subseteq \Autbar(V)$, and one may ask
if $\Autbar(U)_V$ is open in $\Autbar(V)$.

\begin{defi} Let $G$ be a profinite group. An open subgroup $U$ of $G$
will be called {\it $\Aut$-stable } if $\Autbar(U)_V$ is open in $\Autbar(V)$
for every $V$ open in $U$.
We will say that $G$ is $\Aut_2$ if $G$ is $\Aut_1$ and
some open subgroup of $G$ is $\Aut$-stable.
\end{defi}

We will show (see Theorem~\ref{thm:t2}) that a h.c.c.b. profinite group $G$ is $\Aut_2$ if and only if $\Comm(G)_A$ is locally compact.

\subsection{When is $\Comm(G)_A$ Hausdorff?}
\label{ss:t1}

We start by finding equivalent characterizations of the condition $\Aut_1$.

\begin{prop}
\label{prop:aut1equiv}
Let $G$ be a countably characteristically based profinite group.
Then $\TAut(G)$ is closed in $\Aut(G)$ if and only if
there exists an open subgroup $U$ of $G$ such that every element of
$\TAut(G)$ fixes $U$ pointwise.
\end{prop}
\begin{proof} For an open subgroup $U$ of $G$, let $\Aut(G)^U$ denote the subgroup of
$\Aut(G)$ fixing $U$ pointwise. Clearly, $\Aut(G)^U$ is closed for every $U$,
so if $\TAut(G)=\Aut(G)^U$ for some $U$, then $\TAut(G)$ is closed.

Now assume that $\TAut(G)$ is closed, and let $\ca{C}$ be a countable base of
neighborhoods of $1\in G$. Then
\begin{equation}
\label{eq:taut}
\TAut(G)=\textstyle{\bigcup_{V\in\ca{C}}}\Aut(G)^V.
\end{equation}
Since $G$ is characteristically based, $\Aut(G)$ is profinite, and thus Baire's category theorem implies that
$\Aut(G)^V$ is open in $\TAut(G)$ for some open subgroup $V$ of $U$.
In particular, $\Aut(G)^V$ is of finite index in $\TAut(G)$.
Let $\{g_1,\ldots,g_r\}\in\TAut(G)$ be a left transversal of $\Aut(G)^V$
in $\TAut(G)$. By \eqref{eq:taut}, there exists an open subgroup $U\in\ca{C}$
which is contained in $V$ such that $g_1,\ldots,g_r\in\Aut(G)^U$.  Then every element
of $\TAut(G)$ fixes $U$ pointwise.
\end{proof}

\begin{prop}
\label{inducedtop}
Let $G$ be a h.c.c.b. profinite group which is $\Aut_1$. Then
for every open subgroup $U$ of $G$, the quotient topology on
$\Autbar(U)$ coincides with the topology induced from $\Comm(G)_A$.
\end{prop}

The proof of this proposition is based on a well-known property of compact Hausdorff topological spaces
(see \cite[\S I.9.4, Cor.~3]{bou:top}):
\begin{prop} Let $X$ be a set endowed with two topologies $\ca{T}_1$ and $\ca{T}_2$
such that $X$ is compact with respect to $\ca{T}_1$, Hausdorff with respect to $\ca{T}_2$ and
$\ca{T}_1 \supseteq \ca{T}_2$. Then $\ca{T}_1=\ca{T}_2$.
\label{Bourbaki_basic}
\end{prop}

We shall also point out a special case of Proposition~\ref{Bourbaki_basic}:
\begin{cor}
Let $Q$ be a Hausdorff topological space, and let $P$ be a closed subset of $Q$.
Let $\ca{T}$ be some topology on $P$ such that $P$ is compact with
respect to $\ca{T}$ and the inclusion $i:P\to Q$ is continuous with respect to $\ca{T}$.
Then $\ca{T}$ is induced from $Q$.
\label{profinite_basic}
\end{cor}

In addition, we need the following well known fact:
\begin{lem}
Let $G$ be a profinite group, $A$ a closed subgroup of $G$, and $H$ an open subgroup of $A$.
Then there exists an open subgroup $K$ of $G$ such that $H=A \cap K$.
\label{lem:open subgroup}
\end{lem}
\begin{proof}
Recall that every closed subgroup of a profinite group is the intersection of a family of open subgroups.
Thus $H= \cap K_{\alpha}$, where $\{K_\alpha\}$ are open in $G$. Hence $$H=A \cap H=\cap (A \cap K_\alpha).$$
Since $H$ is an open subgroup of $A$, there exists a finite subfamily $\{K_i\}_{i=1}^n$ of $\{K_\alpha\}$
such that $H=\cap_{i=1}^{n} A \cap K_i$. Then $K=\cap_{i=1}^{n} K_i$ is open in $G$ and $H=A \cap K$.
\end{proof}

\begin{proof}[Proof of Proposition~\ref{inducedtop}]
Let $\ca{T}_Q$ be the quotient topology on $\Autbar(U)$ and  $\ca{T}_A$ the topology induced on $\Autbar(U)$ from $\Comm(G)_A$.
Since by definition $\rho_U:\Aut(U) \to \Comm(G)_A$ is continuous, we have $\ca{T}_A \subseteq \ca{T}_Q$.

Let $U=U_1\supset U_2\supset\ldots$ be a super-characteristic base for $U$,
let $A_i=\Autbar(U_i)$ and $A=\cup A_i$. By Proposition~\ref{prop:pureaut},
$A$ is open in $\Comm(G)_A$. Note that each $A_i$ is profinite with respect to the quotient topology,
and the inclusion $A_i\to A_{i+1}$ is continuous with respect to the quotient topologies on $A_i$ and $A_{i+1}$.
We deduce from Corollary~\ref{profinite_basic} that since $A_i$ is compact and $A_{i+1}$ is Hausdorff,
the quotient topology on $A_i$ coincides with the topology induced from $A_{i+1}$.

Let $V$ be a subgroup of $A_1=\Autbar(U)$, such that $V \in \ca{T}_Q$.
By Lemma~\ref{lem:open subgroup} we can construct inductively subgroups $H_i\subset A_i$ such that
\begin{itemize}
\item[(i)] $H_1=V$;
\item[(ii)] $H_i$ is open in $A_i$ (with respect to the quotient topology) for all $i$,
\item[(iii)] $H_{i+1}\cap A_i=H_i$ for all $i$.
\end{itemize}

Let $H=\bigcup H_i$. Then $H$ is open in $\Comm(G)_A$ by Claim~\ref{claim:critopen},
and it is clear from the construction that $H\cap \Autbar (U)=V$. Therefore, $\ca{T}_Q \subseteq \ca{T}_A$.
\end{proof}

\begin{thm}
\label{thm:t1}
Let $G$ be a h.c.c.b. profinite group. Then
$\Comm(G)_A$ is Hausdorff if and only if $G$ is $\Aut_1$.
\end{thm}
\begin{proof} ``$\Rightarrow$'' Suppose that $\Comm(G)_A$ is Hausdorff. Then
the set $\{1\}$ is closed in $\Comm(G)_A$. If $U$
is an open subgroup of $G$, then $\TAut(U)=\rho_U^{-1}(\{1\})$ is
closed in $\Aut(U)$, and therefore $\Comm(G)_A$ is  $\Aut_1$.

``$\Leftarrow$'' Now suppose that $G$ is $\Aut_1$.
Since $\Comm(G)$ is a topological group, it is enough to prove that $G$ is $T_1$.
As in the previous proof, choose a super-characteristic base
$G_1\supset G_2\supset\ldots$ of $G$, let $A_i=\Autbar(G_i)$ and $A=\cup A_i$.

Let $x\neq 1$ be an arbitrary element of $\Comm(G)$. We need to find an open subgroup
$O$ of $\Comm(G)$ such that $x\not\in O$. If $x\not \in A$, we set $O=A$.
Otherwise, $x\in A_i$ for some $i$. Since $A_i$ is Hausdorff, there exists an
open subgroup $V$ of $A_i$ such that $x\not\in V$. By the proof of Proposition~\ref{inducedtop},
$V=A_i\cap O$ for some open subgroup $O$ of $\Comm(G)$, and clearly, $x\not\in O$.
\end{proof}

\subsection{When is $\Comm(G)_A$ locally compact?}
\label{ss:loccom}

\begin{thm}
\label{thm:t2}
Let $G$ be a h.c.c.b. profinite group.
\begin{itemize}
\item[(a)] The following
conditions are equivalent.
\begin{itemize}
\item[(i)] $G$ is $\Aut_2$.
\item[(ii)] $\Comm(G)_A$ is locally compact.
\end{itemize}
\item[(b)] If conditions (i) and (ii) are satisfied, then
$\Comm(G)_A$ is $\sigma$-compact.
\item[(c)] Assume that $\VZ(G)=\{1\}$. Then both (i) and (ii) hold if and only if
there exists an open subgroup $U$ of $G$ such that $[\IC(U):\iota(U)]$ is countable.
\end{itemize}
\end{thm}
\begin{proof}
(a) ``(i)$\Rightarrow$ (ii)'' Assume that $G$ is $\Aut_2$,
and let $U$ be an $\Aut$-stable open subgroup of $G$. We claim that $\Autbar (U)$ is open and
compact in $\Comm(G)_A$, which would mean that $\Comm(G)_A$ is locally compact.

By Proposition~\ref{inducedtop}, the group $\Autbar(U)$ is profinite and therefore compact.
If $V$ is any open subgroup of $U$, then $\Autbar(U)_V$ is open in $\Autbar(V)$, and thus
$\rho_{V}^{-1}(\Autbar(U)_V)$ is open in $\Aut(V)$. Therefore,
$\rho_{V}^{-1}(\Autbar(U))\supseteq \rho_{V}^{-1}(\Autbar(U)_V)$ is also open in
$\Aut(V)$, and thus $\Autbar(U)$ is open in $\Comm(G)_A$.
\vskip .1cm
``(ii)$\Rightarrow$ (i)'' Suppose that $\Comm(G)_A$ is locally compact. In particular,
$\Comm(G)_A$ is Hausdorff, and thus $G$ must be $\Aut_1$ by Theorem~\ref{thm:t1}.

Let $\{\,G_i\}_{i=1}^{\infty}$ be a super-characteristic base of $G$, and let
$A=\bigcup_{i\in\N}\Autbar(G_i)$. By Claim~\ref{claim:critopen}, $A$ is open
and therefore closed in $\Comm(G)_A$. In particular, $A$ is locally compact, and thus
by Baire's theorem there exists some $k\in \N$ such that $\Autbar(G_k)$ is open
in $A$ and thus open in $\Comm(G)_A$.
Let $U= G_k$. If $V$ is any open subgroup of $U$, then $\Autbar(U)_V$ is open
in $\Autbar(U)$, and thus open in $\Comm(G)_A$. In particular, $\Autbar(U)_V$
is open in $\Autbar(V)$. Thus, $U$ is $\Aut$-stable, so $G$ is $\Aut_2$.

(b) Assume that $\Comm(G)_A$ is locally compact, and let $A$ and $U$ be as in the proof of the
implication ``(ii)$\Rightarrow$ (i)''. Then $\Autbar(U)$ is compact, and clearly $\Autbar(U)$ has countable
index in $A$. On the other hand, $A$ has countable index in $\Comm(G)$ by
Proposition~\ref{prop:pureaut}. Thus, $[\Comm(G):\Autbar(U)]$ is countable,
whence $\Comm(G)_A$ is $\sigma$-compact.
\vskip .1cm
\newcommand{\Ubar}{{\overline{U}}}
(c) ``$\Rightarrow$'' Assume that $G$ is $\Aut_2$, and let $U$ be
an $\Aut$-stable open subgroup of $G$. By the same argument as in (b),
the index $[\Comm(G):\Autbar(U)]$ is countable. Let $S$ be a left transversal for
$\Autbar(U)$ in $\Comm(G)$ (thus $S$ is countable as well).

Let $\Ubar=\iota(U)$. By definition, every element of $\IC(U)$ is of the form
$u_1^{s_1}\ldots u_k^{s_k}$ where $u_i\in\Ubar$ and $s_i\in S$ for $1\leq i\leq k$.
To prove that $[\IC (U):\iota(U)]$ is countable, it is sufficient to show that for
fixed $s_1,\ldots, s_k\in S$, the set $\Ubar^{s_1}\ldots \Ubar^{s_k}$ is a covered by finitely many
left cosets of $\Ubar$.

First note that for every $x\in \Comm(G)$ there exists a finite set $T=T(x)$ such that
$\Ubar^x\subseteq T\Ubar$ since $\Ubar^x\cap \Ubar$ is a finite index subgroup of $\Ubar$.
Similarly, given two left cosets $x \Ubar$ and $y \Ubar$ we have
$x \Ubar y \Ubar=xy \cdot \Ubar^y \Ubar\subseteq xy T(y) \Ubar$. The above claim easily follows.
\vskip .15cm

``$\Leftarrow$'' Let $U$ be an open subgroup of $G$ such that $[\IC (U):\iota(U)]$ is countable.

\it{Step 1: }\rm $\iota(U)$ has only countably many conjugates in $\Comm(U)$.

\it{Subproof. }\rm For every $f\in \Comm(U)$ there exists an open normal subgroup
$V$ of $U$ such that $\iota(U)^f\supseteq \iota(V)$. Since $[\IC(U):\iota(V)]$
is countable, there are only countably many possibilities for $\iota(U)^f$ once
$V$ is fixed, and there are only countably many possibilities for $V$ since
$U$ is countably based.
\vskip .12cm

\it{Step 2: }\rm $U$ is $\Aut$-stable.

\it{Subproof. }\rm Step 1 implies that the normalizer of $\iota(U)$ in $\Comm(U)$
has countable index. Note that this normalizer is precisely $\Autbar(U)$.
It follows that for every $V$ open in $U$, the index of $\Autbar(U)_V$ in $\Autbar(V)$
is countable. On the other hand, both $\Autbar(U)_V$ and $\Autbar(V)$ are compact,
so the index of $\Autbar(U)_V$ in $\Autbar(V)$ is either finite
or uncountable. Thus this index has to be finite, so $U$ is $\Aut$-stable,
and $G$ is $\Aut_2$ (note that $G$ is automatically $\Aut_1$ since
$\VZ(G)=\{1\}$).
\end{proof}

\subsection{$\Aut$-topology versus natural topology}
Let $G$ be a profinite group such that $\Comm(G)$ is isomorphic
(as an abstract group) to some ``familiar'' group which
comes with natural topology. As a rule, we expect the $\Aut$-topology
on $\Comm(G)$ to coincide with that natural topology. In this
subsection we shall ``confirm this rule'' in the case of profinite
groups covered by Theorems~\ref{thm:padic} and \ref{Chevalley}.
We shall use the following technical but easy-to-apply criterion.

\begin{prop}
\label{prop:autnatural}
Let $G$ be a h.c.c.b. profinite group which is $\Aut_1$.
Suppose that $\Comm(G)$ is a topological group with respect
to some topology $\ca{T}$, and there exists an open subgroup
$U$ of $G$ such that
\begin{itemize}
\item[(i)] The index $[\Comm(G):\Autbar(U)]$ is countable,
\item[(ii)] $\Autbar(U)$ is an open compact subgroup of $(\Comm(G),\ca{T})$,
\item[(iii)] If $N$ is an open subgroup of $U$ and $\{f_n\}_{n=1}^{\infty}$
is a sequence in $\Autbar(U)$ such that $f_n\to 1$ with respect to $\ca{T}$,
then $f_n(N)=N$ for sufficiently large $n$.
\end{itemize}
Then $U$ is $\Aut$-stable (whence $\Comm(G)_A$ is locally compact), and
$\ca{T}$ coincides with the $\Aut$-topology on $\Comm(G)$.

\end{prop}
\begin{proof}
Recall that $\ca{T}_{A}$ denotes the $\Aut$-topology.
First, from the proof of step 2 in part (c) of Theorem~\ref{thm:t2} we know that
condition (i) implies that $U$ is $\Aut$-stable, and thus
$\Autbar(U)$ is an open compact subgroup of $\Comm(G)_A$.
Condition (iii) can be reformulated as follows: if $\{f_n\}_{n=1}^{\infty}$
is a sequence in $\Autbar(U)$ such that $f_n\to 1$ with respect to $\ca{T}$,
then $f_n\to 1$ with respect to $\ca{T}_{A}$. This implies that
$\ca{T}_{A}$ restricted to $\Autbar(U)$ is not stronger than $\ca{T}$. Since
$(\Autbar(U), \ca{T})$ is compact and $(\Autbar(U), \ca{T}_A)$ is Hausdorff
(as $G$ is  $\Aut_1$), the topologies $\ca{T}$ and $\ca{T}_{A}$
coincide on $\Autbar(U)$ by Proposition~\ref{Bourbaki_basic}.
Since $\Autbar(U)$ is open with respect
to both $\ca{T}$ and $\ca{T}_{A}$, it follows that
$\ca{T}$ and $\ca{T}_{A}$ must coincide on $\Comm(G)$.
\end{proof}

\begin{example}
\label{ex:autnatural}
(a) Let $G$ be a compact $p$-adic analytic group. By Theorem~\ref{thm:padic},
$\Comm(G)$ is isomorphic to $\Aut_{\Q_p}(\eu{L}(G))$, and so $\Comm(G)$ is a subgroup
of $\GL_n(\mathbb Q_p)$ for some $n$. Let $\ca{T}$ be the topology on $\Comm(G)$
induced from the field topology on $\mathbb Q_p$. Conditions (i)-(iii)
of Proposition~\ref{prop:autnatural} are easily seen to hold with $U=G$,
and thus $\ca{T}$ coincides with the $\Aut$-topology.

(b) Let $\dbG$ be a split Chevalley group, let $F$ be a local field of positive characteristic,
and let $G$ be an open compact subgroup of $\dbG(F)$. By Theorem~\ref{Chevalley},
$\Comm(G)\cong \dbG_{ad}(F)\rtimes (X\times \Aut(F))$ where $X$ is the group
of Dynkin diagram automorphisms. Endow $\dbG_{ad}(F)$ and $\Aut(F)$ with their natural topologies,
$X$ with discrete topology, and $\Comm(G)$ with product topology; call this topology $\ca{T}$.
By the same argument as in part (a), $\ca{T}$ coincides with the $\Aut$-topology.
\end{example}

\subsection{Profinite groups which are not $\Aut_2$}
In this subsection we give two examples of profinite groups which
are $\Aut_1$, but not  $\Aut_2$.

\begin{prop}
\label{prop:free}
Let $G$ be a finitely generated free pro-$p$ group of rank $r>1$.
Then $G$ is  $\Aut_1$ but not $\Aut_2$.
\end{prop}

\begin{proof} As $G$ has trivial virtual center, $G$ is $\Aut_1$.
Since an open subgroup of a free pro-$p$ group is free, it suffices
to show that $G$ is not $\Aut$-stable. Furthermore, it will be enough to
show that the image of the map $r_{G,U}\colon\Aut(G)\to\Aut(U)$ is not of finite index
whenever $U$ is an open characteristic subgroup of $G$.

So, assume that $U$ is open and characteristic in $G$.
Let $G^{\ab}$ and $U^{\ab}$ denote the abelianizations of $G$ and $U$, respectively.
Then $\Aut(G)$ acts naturally on $G^{\ab}$ and $\Aut(U)$ acts naturally on $U^{\ab}$.
The transfer $T\colon G^{\ab}\to U^{\ab}$ is an injective map which commutes with the action of
$\Aut(G)$ via the homomorphism $r_{G,U}$ (see \cite[\S10.1]{Rob}), that is,
for $\bg\in G^{\ab}$ and $\alpha\in\Aut(G)$ one has
\begin{equation}
\label{eq:trans}
T(\alpha.\bg)=r_{G,U}(\alpha).T(\bg)
\end{equation}
Let $H=\image(T)$. Then $\rk(H)=\rk(G^{\ab})=r$, and by \eqref{eq:trans},
$\image(r_{G,U})$ leaves $H$ invariant. On the other hand,
if $[G:U]=p^n$, then $U^{\ab}$ is a free $\Z_p$-module of rank $m\deq 1+p^n(r-1)>r$, and
the natural mapping $\Aut(U)\to\Aut(U^{\ab})$ is easily seen to be surjective.
Thus, a finite index subgroup of $\Aut(U)$ cannot stabilize $H$, and therefore
$\image(r_{G,U})$ is not of finite index in $\Aut(U)$.
\end{proof}

\begin{rem}
\label{rem:surface}
The argument used in the proof of Proposition~\ref{prop:free} also applies
in other situations; for instance, it can be used verbatim to show that
the pro-$p$ completion of an orientable surface group of genus $g>0$ is not $\Aut_2$.
\end{rem}

Another series of profinite groups not satisfying the condition
$\Aut_2$ can be found within the class of branch groups. To keep things simple, we discuss
the more restricted class of self-replicating groups.

\begin{defi}\rm A profinite group $G$ is {\it self-replicating,} if
$G$ has trivial center and every open subgroup $K$ of $G$
contains an open subgroup $H$ such that
\begin{itemize}
\item[(i)] $H$ is normal in $G$, and
\begin{equation}
\label{eq:selfrep}
H\simeq \underbrace{G\times G\times \ldots \times G}_{\text{$n$ times }}
\end{equation}
for some $n>1$
\item[(ii)] The conjugation action of $G$ on $H$ permutes
the factors of (\ref{eq:selfrep}) transitively among themselves.
\end{itemize}
\end{defi}

\begin{rem}
It is clear that the virtual center of a self-replicating group
is trivial as well. Thus, every self-replicating group is $\Aut_1$.
\end{rem}

\begin{prop}
\label{prop:selfrep}
Let $G$ be a h.c.c.b. self-replicating profinite group such that
$\Out(G)=\Aut(G)/\Inn(G)$ is infinite. Then $G$ is not $\Aut_2$.
\end{prop}

\begin{proof}
Since $G$ is self-replicating, to prove that $G$ is not $\Aut_2$,
it would be sufficient to show that $G$ is not $\Aut$-stable.

Let $H$ be a normal subgroup of $G$ such that $H=G_1\times G_2\times\ldots \times G_n$,
with $G_i\cong G$ for each $i$, and such that the conjugation action of $G$
on $H$ permutes $G_i$'s transitively.
Given $\phi\in\Aut(G)$, let $\phi_*$ be the automorphism
of $H$ which stabilizes each $G_i$, acts as $\phi$ on $G_1$
and as identity on $G_i$ for $i\geq 2$.

Let $\Aut(G)_H^1=\{\psi\in\Aut(G)\mid \psi(g)\equiv g\mod H\mbox{ for every }g\in G\}$.
Note that $\Aut(G)_H^1\subseteq \Aut(G)_H$, and it is easy to see that
$\Aut(G)_H^1$ is open in $\Aut(G)$. Given $\phi\in \Aut(G)$, we shall
now analyze when $\phi_*$ is equal to $r_{G,H}(\psi)$ for some $\psi\in \Aut(G)_H^1$.

Fix $g\in G$ such that $G_2^g=G_1$.

Suppose that $\psi\in \Aut(G)_H^1$ is such that $r_{G,H}\psi=\phi_*$ for
some $\phi\in \Aut(G)$. Then $\psi(x)=x$ for every $x\in G_2$.
Given $y\in G_1$, we have $y^{g^{-1}}\in G_2$, whence
$$\psi(y)=\psi((y^{g^{-1}})^g)=(y^{g^{-1}})^{\psi(g)}=
y^{g^{-1}\psi(g)}.$$
On the other hand, $g^{-1}\psi(g)\in H$ since $\psi\in \Aut(G)_H^1$.
If $g_1$ is the projection of $g^{-1}\psi(g)$ to $G_1$,
then for every $y\in G_1$ we have
$y^{g^{-1}\psi(g)}=y^{g_1}$, so
$$\psi(y)=y^{g_1}\mbox{ for every } y\in G_1.$$
Thus, the restriction of $\psi$ to $G_1$ is an inner
automorphism.

Now let $A=\{\psi\in\Aut(H) \mid \psi=\phi_*\mbox{ for some }\phi\in\Aut(G)\}$,
and let $r_1:A\to \Aut(G_1)$ be the restriction map.
Let $B=r_{G,H}(\Aut(G)_H^1)\cap A$.
In the previous paragraph we showed that
$r_1(B)$ consists of inner automorphisms of $G_1$. On the other hand,
the map $r_1:A\to \Aut(G_1)$ is clearly surjective.
Since we assume that $\Aut(G_1)$ is not a finite extension
of $\Inn(G_1)$, it follows that $r_{G,H}(\Aut(G)_H^1)$ cannot be a finite
index subgroup of $\Aut(H)$. Hence $G$ is not $\Aut$-stable.
\end{proof}

An example of a group satisfying the hypotheses of
Proposition~\ref{prop:selfrep} is the pro-$2$ completion of the
first Grigorchuk group. This follows from \cite{barsid}.

\subsection{Connection with rigid envelopes}
Although we have argued that the $\Aut$-topology on $\Comm(G)$
is in some sense the natural topology, non-local compactness of $\Comm(G)_A$
tells us fairly little about the possible envelopes of $G$. For instance,
as we showed in Section~\ref{s:tdlc}, the pro-$2$ completion of the Grigorchuk group
has a topologically simple compactly generated envelope, while
its commensurator with the $\Aut$-topology is not locally compact by
Proposition~\ref{prop:selfrep}. Nevertheless, non-local compactness of $\Comm(G)_A$
has the following interesting consequence for envelopes:

\begin{prop}
\label{prop:rigid_rest}
Let $G$ be a h.c.c.b. profinite group with $\VZ(G)=\{1\}$, and assume that
$G$ is not $\Aut_2$. Then $G$ does not have a compactly generated
topologically simple rigid envelope.
\end{prop}
\begin{proof} Assume that $(L,\eta)$ is an envelope for $G$ with the required properties.
Since $L$ is topologically simple and rigid, we have $\eta_*(L)=\IC(G)$
by Corollary~\ref{cor:autcomm}. Since $L$ is compactly generated, it is clear
that $[L:\eta(G)]$ is countable, and therefore $[\eta_*(L):\iota(G)]$
is countable as well. Thus, $G$ is $\Aut_2$ by Theorem~\ref{thm:t2}(c),
contrary to our assumption.
\end{proof}
Combining Proposition~\ref{prop:rigid_rest}, Theorem~\ref{thm:grig}, and Proposition~\ref{prop:selfrep},
we obtain the following interesting result.
\begin{cor}
\label{cor:nonrigid}
There exists a topologically simple compactly generated non-rigid t.d.l.c. group.
\end{cor}

\appendix
\section{}
\label{ss:padic}

In this section we give a proof of Theorem~\ref{thm:padic}
whose statement is recalled below. Our proof is based on
Lazard's exp-log correspondence.

\begin{thm}
Let $G$ be a compact $p$-adic analytic group. Then one has
a canonical isomorphism
\begin{equation}
\label{eq:padic}
\Comm(G)\simeq\Aut_{\Q_p}(\eu{L}(G)).
\end{equation}
where $\eu{L}(G)$ is the Lie algebra of $G$.
\end{thm}

By \cite[Thm.9.31, 9.33]{ddms:padic},
every compact $p$-adic analytic group $G$ contains an open subgroup
which is a torsion-free powerful pro-$p$ group.
An immediate consequence of this fact is that the set
\begin{equation}
\label{eq:defcaP}
\ca{P}_G\colon=\{\,U\in\ca{U}_G\mid\text{$U$ is torsion-free, powerful pro-$p$}\,\}
\end{equation}
is a base of neighborhoods of $1$ in $G$.

There exists a categorical equivalence
between the category $\mathbf{PF}$ of finitely generated torsion-free powerful
pro-$p$ groups and the category $\mathfrak{pf}$ of
powerful $\Z_p$-Lie lattices which is known as {\it Lazard
correspondence}. In other words, there exist functors
\begin{equation*}
\label{eq:laz}
\ca{L}(\argu)\colon \mathbf{PF}\longrightarrow\mathfrak{pf}
\quad\mbox{ and }\quad
\exp(\argu)\colon \mathfrak{pf}\longrightarrow\mathbf{PF}
\end{equation*}
such that the compositions $\ca{L}\circ\exp$ and $\exp\circ\ca{L}$
are naturally isomorphic to the identity functors on the respective categories
\cite[\S8.2]{ddms:padic}.

The {\it Lie algebra} $\eu{L}(G)$ of a compact $p$-adic analytic group $G$
can be defined as follows (see \cite[Chap.V, (2.4.2.5)]{laz:blue}):
\begin{equation}
\label{eq:Lie}
\eu{L}(G)\colon={\textstyle
\varprojlim_{P\in\,\ca{P}_G}\ca{L}(P)\otimes_{\Z_p}\Q_p}.
\end{equation}

\begin{proof}[Proof of Theorem~\ref{thm:padic}]
By definition of $\eu{L}(G)$, for every $P\in\ca{P}_G$ we
can canonically identify $\ca{L}(P)$ with
a $\Z_p$-sublattice of $\eu{L}(G)$.

We shall now construct a  mapping
\begin{equation}
\label{eq:homoLie}
I\colon\Comm(G)\longrightarrow\Aut_{\Q_p}(\eu{L}(G)).
\end{equation}
Let $\phi$ be a virtual automorphism of $G$,
and choose $P\in\ca{P}_G$ such that $\phi$
is defined on $P$. By equivalence of categories,
the isomorphism $\phi: P\to \phi(P)$ corresponds
to an isomorphism $\phi_*: \ca{L}(P)\to\ca{L}(\phi(P))$
which, in turn, uniquely determines an automorphism
$\phi^*: \eu{L}(G)\to \eu{L}(G)$ given
by
\begin{equation}
\label{eq:final}
\phi^*(\alpha x)=\alpha \phi_*(x) \mbox{ for } x\in\eu{L}(G)\mbox{ and }\alpha\in\Q_p.
\end{equation}
Clearly, $\phi^*$
is independent of the choice of $P$, and similarly,
if $[\phi]=[\psi]$, then $\phi^*=\psi^*$.
Thus, we can define
$I\colon\Comm(G)\to\Aut_{\Q_p}(\eu{L}(G))$
by $I([\phi])=\phi^*$.

It is easy to see that $I$ is an injective homomorphism.
In order to prove surjectivity, we have to show that for every
$\gamma\in\Aut_{\Q_p}(\eu{L}(G))$
there exist $P_1, P_2\in\ca{P}_G$, such that
$\gamma(\ca{L}(P_1))=\ca{L}(P_2)$.

Take any $P\in\ca{P}_G$, and let
$L=\ca{L}(P)\,\cap\, \gamma^{-1}(\ca{L}(P))$.
Since both $\ca{L}(P)$ and $\gamma^{-1}(\ca{L}(P))$
are powerful $\Z_p$-Lie lattices, so is their intersection.
Therefore, $L$ is a powerful $\Z_p$-Lie sublattice of $\ca{L}(P)$.
Thus $L=\ca{L}(P_1)$ where $P_1$ is some open torsion-free powerful pro-$p$ subgroup
of $P$, whence $P_1\in\ca{P}_G$.
Furthermore, $\gamma(L)\leq\ca{L}(P)$,
and thus $\gamma(L)=\ca{L}(P_2)$ for some
$P_2\in\ca{P}_G$. Hence $I$ is surjective.
\end{proof}

\bibliography{commtdlc}
\bibliographystyle{amsplain}
\end{document}